\documentclass[9pt,reqno]{amsart}
\allowdisplaybreaks
\usepackage{amsmath,amstext,amssymb,epsfig}
\usepackage{graphicx}
\usepackage{mathrsfs}
\calclayout
\makeatletter
\renewcommand{\@seccntformat}[1]{\bf\csname the#1\endcsname.}
\renewcommand{\section}{\@startsection{section}{1}
 \z@{.7\linespacing\@plus\linespacing}{.5\linespacing}
 {\normalfont\upshape\bfseries\centering}}
\renewcommand{\@biblabel}[1]{\@ifnotempty{#1}{#1.}}
\makeatother

\theoremstyle{plain}
\newtheorem{thm}{Theorem}[section]
\newtheorem{lem}[thm]{Lemma}
\newtheorem{prop}[thm]{Proposition}

\theoremstyle{definition}
\newtheorem{ex}[thm]{Example}
\newtheorem{defn}[thm]{Definition}

\usepackage[parfill]{parskip}
\usepackage[german]{varioref}
\usepackage[all]{xy}
\usepackage{cancel}

\usepackage{color}
\usepackage{tikz-cd}
\def\R{{\mathcal R}}
\def\N{{\mathcal N}}

\def\T{{\mathcal T}}
\def\L{{\mathcal L}}
\def\>{\succ}
\def\<{\prec}

\def\M{{\mathcal M}}

\def\b{\beta}
\def\a{\alpha}
\def\lt{\triangleleft}
\def\rt{\triangleright}
\def\l{\lambda}
\def\p{\partial}

\def\m{\mu}

\begin{document}
	\title[Sania Asif\textsuperscript{1}
 ]{Exploring Cohomology, Deformations, and Hom-NS Structures in Hom-Leibniz Conformal Algebras through Nijenhuis Operators}
\author{Sania Asif\textsuperscript{1}
 }
	\address{\textsuperscript{1}Institute of Mathematics, Henan Academy of Sciences, Zhengzhou 450046, P. R. China.}
	\email{\textsuperscript{1}11835037@zju.edu.cn, 200036@nuist.edu.cn }
	\date{\today}
    \begin{abstract} 
    This paper studies the Nijenhuis operator on Hom-Leibniz conformal algebra, defining their representations and cohomologies. We determine the cohomologies for both Hom-Leibniz conformal algebra and Nijenhuis operators on Hom-Leibniz conformal algebra. Subsequently, establishing the cohomology of Hom-Nijenhuis-Leibniz conformal algebras. As an application to this cohomology, we study formal deformations of the Nijenhuis operator on Hom-Leibniz conformal algebra. Additionally, we introduce Hom-NS-Leibniz conformal algebra and explore how various operators such as Rota-Baxter operator, Twisted Rota Baxter operator, and Nijenhuis operators can provide Hom-NS-Leibniz conformal algebras.\end{abstract}
\footnote{Author would like to thank the Editor and referees for their valuable input in our manuscript.}
	\maketitle
	\section{ Introduction}
    Over the past fifteen years, there has been significant interest in algebras twisted by homomorphisms, known as Hom-type algebras. Hartwig, Larsson, and Silvestrov introduced the concept of Hom-Lie algebras in their work on $q$-deformations of the Witt and Virasoro algebras. Subsequently, various Hom-type algebras, including Hom-associative, Hom-Leibniz, and Hom-Poisson algebras, have been extensively studied, see \cite{AW2, MS, Yuan, ZYC}. Theories about Hom-associative conformal algebras and Hom-Lie conformal algebras have been developed, along with formal deformation theories for Hom-associative and Hom-Lie algebras in \cite{AW, AWW}. These structures can deform according to the Hom-analog of the Chevalley-Eilenberg cohomology. 
    \par Leibniz algebra, introduced by Bloh in \cite{B} and later re-examined by Loday in \cite{L}, generalizes Lie algebra by relaxing the skew-symmetry requirement. It has applications in geometry, physics, and homotopy theory. Leibniz conformal algebra extends this concept to the conformal setting and is closely related to field algebras, which generalizes vertex operator algebras in a non-commutative context. The Leibniz conformal algebra was first considered by Bakalov, Kac and Voronov \cite{BKV} as the non-skew-symmetric analogue of Lie conformal algebras. Subsequently, \cite{Zhang} introduces the cohomology of Leibniz conformal algebras, and \cite{Wu} developed the concept of a Leibniz pseudoalgebra which is a multivariable generalization of Leibniz conformal algebras. These results extend to the classification of torsion-free Leibniz conformal algebra of rank three in \cite{Wu2}. The extension theory of Leibniz conformal algebras is examined in \cite{HY}, while \cite{ZH} studies quadratic Leibniz conformal algebras and their central extensions as a special class of Leibniz conformal algebras. Beyond individual studies, Leibniz conformal algebras are also explored for operators; \cite{GW} examines Twisted relative Rota-Baxter operators on Leibniz conformal algebras. Most recently, \cite{DS}  explores the homotopification of Leibniz conformal algebras. Additional research in this aspect is available in \cite{WGZ}.
\par The concept of a Nijenhuis operator was introduced by Albert Nijenhuis in the mid-20th century in \cite{NA}. Nijenhuis's work focused on differential forms and the integrability conditions of almost complex structures, leading to the definition of what is now known as the Nijenhuis tensor. This tensor measures the failure of an almost complex structure to be integrable. An operator is termed a Nijenhuis operator if its associated Nijenhuis tensor vanishes, indicating integrability. Algebraically, For an algebra $\L$ with multiplication $\circ:\L\otimes\L \to \L,$ we can define Nijenhuis operator as an operator $\N:\L \to \L$, such that following equation hold
$$\N(p) \circ \N (q)  = \N ( \N (p)\circ q +  p\circ \N (q) - \N  (p\circ q) ), \quad \forall~ p,q\in \L. $$
The study of Nijenhuis operators has significantly influenced various mathematical fields, specifically complex differential geometry and algebraic deformation theory. The renowned Newlander-Nirenberg theorem asserts that an almost complex structure on a manifold (or Lie algebra) qualifies as a complex structure if and only if it is a Nijenhuis operator \cite{NN}. In algebraic deformation theory, as explored by Gerstenhaber \cite{G}, a Nijenhuis operator on an associative or Lie algebra facilitates a straightforward linear deformation of the respective algebra. Recent developments on the Nijenhuis operator on various (conformal) algebras and their cohomology and deformations can be explored in \cite{AW, AWW, LSZB, BR, WSBL}.
\par Motivated by the significance of Leibniz conformal algebras, Nijenhuis operators, and Hom-algebras, we delve into the intricate aspects of their cohomology and deformation structures. This paper focuses on studying Nijenhuis operators within the framework of Hom-Leibniz conformal algebras. Although many researches on Nijenhuis operator and Leibniz conformal algebras as individual studies, no research has yet been conducted about Nijenhuis operator on Leibniz conformal algebras and Hom-Leibniz conformal algebras. Studying Nijenhuis operators within the framework of Hom-Leibniz conformal algebras can provide us results about Nijenhuis operators on the Leibniz conformal algebras by keeping twist map $\a= id$. That is why this study is quite important and aids in filling up many gaps for future research.
Building upon further discussion, the goal of this paper is not only to construct the cohomology of the Nijenhuis operator on Hom-Leibniz conformal algebras but also the cohomology of Hom-Nijenhuis Leibniz conformal algebras. We first study the Nijenhuis operator on Hom-Leibniz conformal algebra. We define the representation and cohomology associated with a Hom-Leibniz conformal algebra and Hom-Nijenhuis operators individualistically; their corresponding coboundary operators are expressed by $\delta_{HomL}$ and $\p_{HN}$, respectively can be seen in the textual body of this paper. Later we use the cohomology of Hom-Leibniz conformal algebra and cohomology of Nijenhuis operator to determine the cohomology of Hom-Nijenhuis Leibniz conformal algebras. As an application to cohomology, we study formal deformations of the Nijenhuis operator on Hom-Leibniz conformal algebra. We introduced Hom-NS-Leibniz conformal algebra as the Hom-analog of $NS$-Lie conformal algebra and comprehended how various operators can yield a Hom-NS-Leibniz conformal algebras. Hom-$NS$-Leibniz conformal algebras relate to twisted relative Rota-Baxter operators in the same way that Hom-Leibniz dendriform conformal algebras are related to relative Rota-Baxter operators. Nijenhuis operator on a Hom-Leibniz conformal algebra naturally induces a Hom-$NS$-Leibniz conformal algebra structure. Finally, we introduce twisted Rota-Baxter operators on Hom-Leibniz conformal algebras, which also induces Hom-$NS$-Lie conformal algebra structures using the Nijenhuis operator. For more about $NS$-algebras, see \cite{AW, AWW, GW}.
\par This paper is organized as follows; in Section $2$, we provide preliminaries about Nijenhuis operators and discuss the representations of Hom-Nijenhuis Leibniz conformal algebra. In Section $3$, we focus on defining the cohomology of a Hom-Nijenhuis-Leibniz conformal algebra. Section $4$ is devoted to finding the deformation of Hom-Nijenhuis Leibniz conformal algebra. Finally, Section $5$ studies Hom-NS-Leibniz conformal algebra and Twisted-Rota-Baxter operators.
\par Throughout this paper all vector spaces, tensor products, and (bi-)linear maps are over a field $\mathbb{K}$ of characteristic~$0$.
\section{ Representations of Hom-Nijenhuis Leibniz conformal algebra}
	In this section, we present some preliminary results to strengthen our findings in the later sections. \begin{defn}\label{defdef} 
	A $\mathbb C[\p]$-module $\L$ is called a Hom-Leibniz conformal algebra $\L$, if it is equipped with a $\mathbb C$-bilinear map $\L \otimes \L \to \L[\l ], p\otimes q \mapsto [p_\l q]$, called the $\l$-bracket, and a linear map $\a:\L\to\L$ that satisfies the following axioms
	\begin{equation}
	\begin{aligned}
	\a\p&=\p\a\\
	\a[p_\l q]&=[\a(p)_\l\a(q)], &\text{(Multiplicativity)}\\
	[\p p_\l q] &= -\l [p_\l q], [p_\l \p (q)] = (\p +\l )[p_\l q], &\text{(Conformal-sesqui-linearity)} \\ [\a(p)_\l [q_\m r]] &= [[p_\l q]_{\l+\m}\a(r)]+[\a(q)_{\m} [p_\l r]], & \text{(Hom-Leibniz conformal identity)} 
	\end{aligned}
	\end{equation}
	for all $p,q,r\in \L.$
\end{defn} If Hom-Leibniz conformal algebra satisfies an additional condition of conformal skew-symmetry, given by $$[p_\l q]= -[{q}_{-\l-\p} p], ~~~~~~~~~\text{(skew-symmetry)}$$ then we call it a Hom-Lie conformal algebra, the obtained Hom-Lie conformal algebra is called Lie conformal algebra if twist map $\a=Id$. Similar to a Lie conformal algebra, a Leibniz conformal algebra is called finite if it is finitely generated $\mathbb{C}[\p]$-module and vice versa. Moreover, all Hom-Lie conformal algebras are Hom-Leibniz conformal algebras.
\begin{ex}
	Consider a $\mathbb{C}[\p]$-module $\L$ be a Leibniz conformal algebra of rank $1$. Then $\L$ is isomorphic to Virasoro Lie conformal algebra satisfying $[\L_\l \L]= (\p+ 2 \l) \L$.
\end{ex}
\begin{ex}
   Let $\L$ be a Hom-Leibniz algebra with Lie bracket $[\cdot, \cdot]$. Let $Cur(\L):= \mathbb{C}[\p]\otimes \L$ be the $\mathbb{C}[\p]$-module. Then it is a Hom-Leibniz conformal algebra with $\l$-bracket given by
\begin{align*}
    \a(f(\p) \otimes  p) &= f(\p) \otimes  \a(p),\\
[(f(\p) \otimes  p)_\l (g(\p) \otimes  q)] &= f(-\l)g(\p+\l) \otimes  [p,q], \quad \forall ~p,q \in \L.
\end{align*}
\end{ex}
  \begin{defn}\label{Nijenhuis}A $\mathbb{C}[\p]$-linear map $\N_{h} : \L \to \L$ is called a Nijenhuis operator on Hom-Leibniz conformal algebra $(\L,[\cdot_\l\cdot],\a)$, if following axioms hold
 \begin{equation}\begin{aligned}
		\a\circ \N_{h} &=\N_{h}\circ \a,\\
		[\N_{h}(p)_{\l} \N_{h}(q)]&= \N_{h}([\N_{h}(p)_{\l}q]+ [p_{\l}\N_{h}(q)]- \N_{h}[p_{\l}q]), \quad \forall~ p,q\in \L, ~~\l\in \mathbb{C}. \end{aligned}\end{equation}\end{defn} Note that the identity map on $\L$ always satisfies Nijenhuis operator identity.
\begin{defn}\label{Rota-Baxter}A $\mathbb{C}[\p]$-linear map $\R : \L \to \L$ is called a Rota-Baxter operator of weight $\theta\in \mathbb{F}$ on Hom-Leibniz conformal algebra $(\L,[\cdot_\l\cdot],\a)$, if following axioms hold\begin{equation}\begin{aligned}
	\a\circ \R &=\R\circ \a,\\
	[\R(p)_{\l} \R(q)]&= \R([\R(p)_{\l}q]+ [p_{\l}\R(q)]+ \theta[p_{\l}q]), \quad \forall~ p,q\in \L. \end{aligned}\end{equation}
    \end{defn} 
 \begin{defn}\label{modified Rota-Baxter}A $\mathbb{C}[\p]$-linear map $\R : \L \to \L$ is called a modified Rota-Baxter operator of weight $\theta\in \mathbb{C}$ on Hom-Leibniz conformal algebra $(\L,[\cdot_\l\cdot], \a)$, if following axioms hold\begin{equation}\begin{aligned}
		\a\circ \R &=\R\circ \a,\\
		[\R(p)_{\l} \R(q)]&= \R([\R(p)_{\l}q]+ [p_{\l}\R(q)])+ \theta[p_{\l}q], \quad \forall~ p,q\in \L. \end{aligned}\end{equation}
        \end{defn} 
\begin{defn}
	A Hom-Leibniz conformal algebra $(\L,[\cdot_\l\cdot],\a)$ when equipped with the Nijenhuis operator $\N_{h}$, is called as Hom-Nijenhuis-Leibniz conformal algebra. We denote such algebra by $(\L,[\cdot_\l\cdot],\a,\N_{h})$ or simply by $\L_{\N_h}$.
\end{defn}
\begin{defn}
 Let $\L_{\N_h}=(\L,[\cdot_\l\cdot],\a,\N_{h})$ and $\L_{\N_h}'=(\L',[\cdot_\l\cdot]',\a',\N_{h}')$ are two Hom-Nijenhuis-Leibniz conformal algebras, a morphism between them is a map $\phi: \L_{\N_h}\to \L_{\N_h}'$, such that $\phi\circ \N_{h}=\N_{h}'\circ \phi$. 
\end{defn}
\begin{prop}\label{prop2.9}
	Let $\L_{\N_h}$ be a Hom-Nijenhuis-Leibniz conformal algebra by defining a new Leibniz conformal bracket 
    $$[p_\l q]_{\N_h}=[\N_{h}(p)_{\l}q]+ [p_{\l}\N_{h}(q)]- \N_{h}[p_{\l}q], \quad \forall ~~p,q\in \L,$$ we can obtain a new Hom-Nijenhuis-Leibniz conformal algebra $(\L_{\N_h},[\cdot_\l\cdot]_{\N_h},\a)$. Moreover, $\N_{h}$ also acts as a Nijehnuis operator on $(\L_{\N_h},[\cdot_\l\cdot]_{\N_h},\a)$. In this case the map $\N_{h}: (\L_{\N_h},[\cdot_\l\cdot]_{\N_h},\a) \to (\L_{\N_h},[\cdot_\l\cdot],\a)$ is a morphism of Hom-Nijenhuis Leibniz conformal algebra.
\end{prop}
\begin{proof}There are three axioms to prove in this proposition, given as follows.
	\begin{enumerate}
		\item We show the identities of Hom-Leibniz conformal algebras one by one.
	\begin{itemize}
		 \item[(a)] Conformal sesqui-linearity: \begin{align*}
			[\p p_\l q]_{\N_h} &= [\N_{h}(\p p)_{\l}q]+ [\p p_{\l}\N_{h}(q)]- \N_{h}[\p p_{\l}q]\\& = [\p\N_{h}(p)_{\l}q]+ [\p p_{\l}\N_{h}(q)]- \N_{h}[\p p_{\l}q]\\&
			= -\l([\N_{h}(p)_{\l}q] +[ p_{\l}\N_{h}(q)]- \N_{h}[p_{\l}q]),
			\end{align*} similarly we can show $[p_\l \p q]_{\N_h}= (\p +\l )[p_\l q]_{\N_h}.$
				\item[(b)]Conformal multiplicativity: \begin{equation*}
			\begin{aligned}
			\a{[p _{\l}q]}_{\N}=&\a[\N_{h}(p)_{\l}q]+ \a[p_{\l}\N_{h}(q)]- \a\N_{h}([p_{\l}q])\\=&[\N_{h}(\a(p))_{\l} \a(q)]+ [\a(p)_{\l}\N_{h}(\a(q))]- \N_{h}([\a(p)_{\l}\a(q)])\\=& {[\a(p)_{\l}\a(q)]}_{\N}.
			\end{aligned}
			\end{equation*}\item[(c)] Hom-Leibniz conformal identity:
			\begin{align*}
			&[\a(p)_\l [q_\m r]_{\N_h}]_{\N_h} \\=& [\a(p)_\l ([\N_{h}(q)_{\l}r]+ [q_{\l}\N_{h}(r)]- \N_{h}[q_{\l}r])]_{\N_h} \\=& [\a(p)_\l [\N_{h}(q)_{\m}r]]_{\N_h}+ [\a(p)_\l [q_{\m}\N_{h}(r)]]_{\N_h}- [\a(p)_\l \N_{h}[q_{\m}r]]_{\N_h}\\=& [\N_{h}(\a(p))_\l [\N_{h}(q)_{\m}r]]+[\a(p)_\l \N_{h}([\N_{h}(q)_{\m}r])]-\N_{h}([\a(p)_\l [\N_{h}(q)_{\m}r]])\\&+[\N_{h} \a(p)_\l [q_{\m}\N_{h}(r)]]+[\a(p)_\l \N_{h}([q_{\m}\N_{h}(r)])]- \N_{h}([\a(p)_\l [q_{\m}\N_{h}(r)]]) \\&- [\N_{h} \a(p)_\l \N_{h}([q_{\m}r])]- [\a(p)_\l \N_{h}^2 ([q_{\m}r])]+\N_{h}([\a(p)_\l \N_{h}([q_{\m}r])])\\=& 
			[\N_{h}(\a(p))_\l [\N_{h}(q)_{\m}r]]+[\a(p)_\l ([\N_{h}(q)_{\m}\N_{h}(r)])]-\N_{h}([\a(p)_\l [\N_{h}(q)_{\m}r]])\\&+[\N_{h} \a(p)_\l [q_{\m}\N_{h}(r)]]- \N_{h}([\a(p)_\l [q_{\m}\N_{h}(r)]]) \\&- [\N_{h} \a(p)_\l \N_{h}([q_{\m}r])]+\N_{h}([\a(p)_\l \N_{h}([q_{\m}r])]).
			\end{align*}
			Similarly, we have 
			\begin{align*}
		&	[{[p_\l q]_{\N_h}}_{\l+\m}\a(r)]_{\N_h}\\=& [{[\N_{h}(p)_\l \N_{h}(q)]}_{\l+\m}\a(r)]
			+[{[p_\l \N_{h}(q)]_{\N_h}}_{\l+\m}\N_{h}(\a(r))]-\N_{h}[{[p_\l \N_{h}(q)]}_{\l+\m}\a(r)]\\&+[{[\N_{h}(p)_\l q]}_{\l+\m}\N_{h}(\a(r))]-\N_{h}[{[\N_{h}(p)_\l q]}_{\l+\m}\N_{h} \a(r)]\\&- [\N_{h}{[p_\l q]}_{\l+\m}\N_{h}(\a(r))]+ \N_{h}[\N_{h}{[p_\l q]}_{\l+\m}\a(r)],
			\end{align*}
			and \begin{align*}
			&[\a(q)_{\m}[p_\l r]_{\N_h}]_{\N_h}\\=&[\N_{h}(\a(q))_\m [\N_{h}(p)_{\l}r]]+[\a(q)_\m ([\N_{h}(p)_{\l}\N_{h}(r)])]-\N_{h}([\a(q)_\m [\N_{h}(p)_{\l}r]])\\&+[\N_{h} \a(q)_\m [p_{\l}\N_{h}(r)]]- \N_{h}([\a(q)_\m [p_{\l}\N_{h}(r)]]) \\&- [\N_{h} \a(q)_\m \N_{h}([p_{\l}r])]+\N_{h}([\a(q)_\m \N_{h}([p_{\l}r])]).
			\end{align*} From the above equations we conclude that Hom-Leibniz conformal identity holds.
		\end{itemize}Hence $(\L_{\N_h}, [\cdot_\l \cdot]_{\N_h},\a)$ is a Hom-Leibniz conformal algebra.
	\item For any $p,q\in \L$, we have\begin{align*}
	&[\N_{h}(p)_\l\N_{h}(q)]_{\N_h} \\=& [\N_{h}(\N_{h} (p))_\l \N_{h} (q)] + [\N_{h} (p)_\l \N_{h}(\N_{h} (q))] - \N_{h}([\N_{h} (p)_\l \N_{h} (q)]) \\
	=& \N_{h}([\N_{h}(\N_{h} (p))_\l q] + [\N_{h} (p)_\l \N_{h} (q)] - \N_{h}([\N_{h} (p)_\l q])) \\
	&+ \N_{h}([\N_{h} (p)_\l \N_{h} (q)] + [p_\l \N_{h}(\N_{h} (q))] - \N_{h}([p_\l \N_{h} (q)])) \\
	&- \N^2_h([\N_{h} (p)_\l q] + [p_\l \N_{h} (q)] - \N_{h}([p_\l q])) \\
	=& \N_{h} \left([\N_{h}(p)_\l q]_{\N_h} + [p_\l \N_{h} (q)]_{\N_h} - \N_{h}([p_\l q]_{\N_h}) \right).
	\end{align*} Hence, $ \N_{h}$ is also a Nijenhuis operator on the Hom-Leibniz conformal algebra $(\L_{\N_h}, [\cdot_\l \cdot ]_{\N_h})$.
	\item It is obvious to observe that the map $\N_{h} : (\L_{\N_h} , [ \cdot_\l \cdot ]_{\N_h},\a) \to (\L_{\N_h} , [\cdot_\l \cdot],\a)$, $p\mapsto \N_h(p)$ is a morphism of Hom-Nijenhuis Leibniz conformal algebra. \end{enumerate}It completes the proof.
	\end{proof} 
The following proposition describes the close relationship between the Nijenhuis operator and the Rota-Baxter operator. Similar results hold for Hom-associative conformal algebras (see prop 6.2 of \cite{AWW}). Let us observe, if these results hold for Hom-Leibniz conformal algebras or not.
\begin{prop}Let $\N_{h} :\L \to \L$ be a $\mathbb{C}[\p]$-module homomorphism over a Hom-Leibniz conformal algebra $\L$. Then
\begin{enumerate}
 \item  $\N_{h}$ is a Nijenhuis operator if and only if $\N_{h}$ is a Rota-Baxter-operator of weight $0$, given that $\N_{h}^2= 0$ .
		\item $\N_{h}$ is a Nijenhuis operator if and only if $\N_{h}$ is a Rota-Baxter-operator of weight $-1$,  given that $\N_{h}^2 = \N_{h}$.
	\item  $\N_{h}$ is a Nijenhuis operator if and only if $\N_{h}$ is a modified $ RB $-operator of weight $\mp 1,$ given that $\N_{h}^2= \pm Id$.
\item $\N_{h}$ is a Nijenhuis operator if and only if $\N_{h} \pm Id$ is a $RB$-operator of weight $\mp 2,$ given that $\N_{h}^2 = Id$.
	\end{enumerate}
\end{prop}
\begin{proof}
\begin{enumerate}
	\item  Suppose $\N_{h}$ is a Nijenhuis operator and let $\N_{h}^2 = 0$. Then for any $p,q \in \L$, we have
\begin{align*}
	[\N_{h}(p)_\l \N_{h}(q)]& = \N_{h}([\N_{h}(p)_\l q] + [p_\l \N_{h}(q)] - \N_{h}[p_\l q])\\&
	= \N_{h}([\N_{h}(p)_\l q] + [p_\l \N_{h}(q)]) - \N_{h}^2[p_\l q]\\&
	= \N_{h}([\N_{h}(p)_\l q] + [p_\l \N_{h}(q)]).
	\end{align*}
		Hence, $\N_{h}$ is a Rota-Baxter operator of weight $0$.
	\item If $\N_{h}^2=\N_{h}$, then we get 
    \begin{align*}
	[\N_{h}(p)_\l \N_{h}(q)]= \N_{h}([\N_{h}(p)_\l q] + [p_\l \N_{h}(q)]- [p_\l q]).
	\end{align*} Here $\N_{h}$ is nothing but a Rota-Baxter operator of weight $-1$.
	\item If $\N_{h}^2=\pm Id$, then we get 
	\begin{align*}
	[\N_{h}(p)_\l \N_{h}(q)]= \N_{h}([\N_{h}(p), q] + [p, \N_{h}(q)])\mp [p_\l q].
	\end{align*}which is modified Rota-Baxter operator of weight $\mp 1$. 
	\item Consider that $\N_{h}+ Id$ is a Rota-Baxter-operator of weight $-2$, we have\begin{equation*}\begin{aligned}&[(\N_{h}+ Id)(p)_{\l}(\N_{h}+ Id)(q)]- (\N_{h}+ Id)([(\N_{h}+ Id)(p)_\l q ]+ [p_\l (\N_{h}+ Id)(q)]-2([p_\l q]))\\=&
	[\N_{h}(p)_{\l}(\N_{h}+ Id)(q)]+[p_{\l}(\N_{h}+ Id)(q)]\\&- \N_{h}([(\N_{h}+ Id)(p)_\l q] + [p_\l (\N_{h}+ Id)(q)]- 2([p_\l q]))\\&- ([(\N_{h}+ Id)(p)_\l q] + [p_\l (\N_{h}+ Id)(q)] -2([p_\l q]))\\=& \{[\N_{h}(p)_{\l}\N_{h}(q)]+ [\N_{h}(p)_{\l}q]+ [p_{\l}\N_{h}(q)]+ [p_{\l}q]\}\\&- \N_{h}(\{[\N_{h}(p)_\l q]+ [p_{\l}q]+ [p_\l \N_{h}(q)] +[p_{\l}q]-2([p_\l q])\})\\&- (\{[\N_{h}(p)_\l q]+[p_{\l}q] +[ p_\l \N_{h}(q)]+[p_{\l}q]\} -2([p_\l q])))\\=&[\N_{h}(p)_{\l}\N_{h}(q)]- \N_{h}([\N_{h}(p)_\l q]+[ p_\l \N_{h}(q)])+([p_\l q])\\=&[ \N_{h}(p)_{\l}\N_{h}(q)]- \N_{h}([\N_{h}(p)_\l q]+ [p_\l \N_{h}(q)]) +\N_{h}^2([p_\l q])\\=&[ \N_{h}(p)_{\l}\N_{h}(q)]- \N_{h}([\N_{h}(p)_\l q]+ [p_\l \N_{h}(q)]-\N_{h}([p_\l q])).\end{aligned}\end{equation*}This implies that $\N_{h}$ is a Nijenhuis operator with $\N_{h}^2= Id$. Similarly, we can prove other case by considering $\N_{h}-Id$ is the Rota-Baxter-operator of weight $+2$.
\end{enumerate} This completes the proof.\end{proof}
\begin{defn}\label{defrep}
 A representation of a Hom-Leibniz conformal algebra $(\L,[-_\l-],\a)$ consists of a $\mathbb{C}[\p]$-module $\M$ equipped with two $\mathbb{C}$-bilinear maps $l:\L\otimes \M\to \M[\l]$ and $r:\M\otimes\L\to \M[\l]$, defined by $l(p,m)\mapsto l(p)_\l m$ and $r(m,p)\mapsto r(m)_\l p$ respectively and a $\mathbb{C}$-linear map $\b:\M \to \M$, such that following equations hold for all $p,q\in \L$ and $m\in \M$
 \begin{align}
 l(\p p)_\l m = -\l l(p)_\l m,~& l(p)_\l(\p m) = (\p+\l)l(p)_\l m,\\
 r(\p m)_\l p = -\l r(m)_\l p, ~&r (m)_\l(\p p) = (\p+\l)r(m)_\l p\\
 r(\b(m)) _\l ([p_\m q])&= r(r(m)_\l p)_{\l+\m} \a(q) + l(\a(p))_\m (r(m)_\l q),\label{eq3}\\
 r(\b(m))_\l([p_\m q])&= -r(l(p)_\m m)_{\l+\m} \a(q)+ l(\a(p))_\m (r(m)_{\l} q)\label{eq4} ,\\
 l(\a(p))_\l (l(q)_\m m) &= l([p_\l q])_{\l+\m} \b(m) + l (\a(q))_\m (l(p)_\l m).\\
 \b(l(p)_\l m) &= l(\a(p))_\l \b(m),\\
 \b(r(p)_\l m) &= r(\a(p))_\l \b(m).
 \end{align}
\end{defn}
\begin{defn}\label{defrepNIJ}
	A representation of a Hom-Nijenhuis-Leibniz conformal algebra $(\L, [\cdot_\l \cdot],\a,\N_{h})$ is quardruple $(\M, r, l, \b, \N_{\M})$, where $(\M, r, l, \b)$ is a representation of Hom-Leibniz conformal algebra and $\N_{\M}:\M\to \M$ is a $\mathbb{C}$-linear map satisfying the following equations :\begin{align*}
		\b\circ \N_{\M} &=\N_{\M}\circ \b,\\
	l(\N_{h}(p))_{\l} \N_{\M}(m)&= \N_{\M}(l(\N_{h}(p))_{\l}m+ l(p)_{\l}\N_{\M}(m)- \N_{\M} \circ l(p)_{\l}m), \quad \forall~ p\in \L ~and ~m\in\M\\
		r(\N_{\M}(m))_{\l} \N_{h}(p)&= \N_{\M}(r(\N_{\M}(m))_{\l}p+ r(m)_{\l}\N_{h}(p)- \N_{\M} \circ r(m)_{\l}p), \quad \forall~ p\in \L ~and ~m\in\M.
	\end{align*}
\end{defn}
\begin{prop}\label{prop2.12}
	Let $\L_{\N_h}=(\L, [\cdot_\l\cdot],\a,\N_{h})$ be a Hom-Nijenhuis-Leibniz conformal algebra with the representation $(\M, l , r, \b, \N_{\M} ).$ We define $l'$ and $ r'$ respectively by 
	\begin{align*}
	 l'(p)_\l m&=l(\N_{h}(p))_{\l}m+ l(p)_{\l}\N_{\M}(m)- \N_{\M} \circ l(p)_{\l}m, \\ 
	r'(m)_\l p&= r(\N_{\M}(m))_{\l}p+ r(m)_{\l}\N_{h}(p)- \N_{\M} \circ r(m)_{\l}p.
	\end{align*}Then $(\M, l' , r', \N_{\M},\b )$ is a representation of $(\L_{\N_h}, [\cdot_\l \cdot]_{\N_h}, \a)$.
\end{prop}
\begin{proof}Here $ l' $ and $ r' $ need to satisfy the conditions given in Definition \eqref{defrep}: 
	\begin{enumerate}
	 \item \begin{align*}
			l'(\p p)_\l m&=l(\N_{h}(\p p))_{\l}m+ l(\p p)_{\l}\N_{\M}(m)- \N_{\M} \circ l(\p p)_{\l}m\\ &=
			l(\p \N_{h}(p))_{\l}m+ l(\p p)_{\l}\N_{\M}(m)- \N_{\M} \circ l(\p p)_{\l}m\\ &= 
			-\l l(\N_{h}(p))_{\l}m-\l l(p)_{\l}\N_{\M}(m)+ \l \N_{\M}\circ l(p)_{\l}m\\ &= -\l (l(\N_{h}(p))_{\l}m+l(p)_{\l}\N_{\M}(m)- \N_{\M}\circ l(p)_{\l}m)\\ &= -\l (l'(p)_\l m).
	\end{align*}Similarly, we can show that $l'(p)_\l(\p m) = (\p+\l)l'(p)_\l m.$
 \item \begin{align*}
		 r' (m)_\l(\p p) &= r(\N_{\M}(m))_{\l}(\p p)+ r(m)_{\l}\N_{h}(\p p)- \N_{\M} \circ r(m)_{\l}(\p p) \\=&(\p+\l)r(\N_{\M}(m))_{\l}(p)+ (\p+\l)r(m)_{\l}\N_{h}(p)-(\p+\l) \N_{\M} \circ r(m)_{\l}(p) \\=&(\p+\l)(r(\N_{\M}(m))_{\l}(p)+ r(m)_{\l}\N_{h}(p)- \N_{\M} \circ r(m)_{\l}(p) )\\=&(\p+\l)r'(m)_\l p.  \end{align*}
		 \item  \begin{align*}
		 &l'(\a(p))_\l (l'(q)_\m m)- l'([p_\l q]_{\N_h})_{\l+\m} \b(m) - l' (\a(q))_\m (l'(p)_\l m) \\=& l(\N_{h}(\a(p)))_{\l}(l'(q)_\m m)+ l(\a(p))_{\l}\N_{\M}(l'(q)_\m m)- \N_{\M} \circ l(\a(p))_{\l}(l'(q)_\m m)\\&- l(\N_{h}([p_\l q])_\N)_{\l+\m} \b(m)- l([p_\l q]_{\N_h})_{\l+\m}\N_{\M}(\b(m))+ \N_{\M} \circ l([p_\l q]_{\N_h})_{\l+\m}\b(m)
 \\&- l(\N_{h}(\a(q)))_{\m}(l'(p)_\m m)+ l(\a(q))_{\l}\N_{\M}(l'(p)_\l m)- \N_{\M} \circ l(\a(q))_{\m}(l'(p)_\l m)\\=&l(\N_{h}(\a(p)))_{\l}(l(\N_{h}(q))_{\m}m+ l(q)_{\m}\N_{\M}(m)- \N_{\M} \circ l(q)_{\m}m)\\&
 + l(\a(p))_{\l}\N_{\M}(l(\N_{h}(q))_{\m}m+ l(q)_{\m}\N_{\M}(m)- \N_{\M} \circ l(q)_{\m}m)\\&
 - \N_{\M} \circ l(\a(p))_{\l}(l(\N_{h}(q))_{\m}m+ l(q)_{\m}\N_{\M}(m)- \N_{\M} \circ l(q)_{\m}m) \\&
 - l(([\N_{h} (p)_\l \N_{h} (q)]))_{\l+\m} \b(m)\\&
 - l([\N_{h}(p)_{\l}q]+ [p_{\l}\N_{h}(q)]- \N_{h}[p_{\l}q])_{\l+\m}\N_{\M}(\b(m))\\&
 + \N_{\M} \circ l([\N_{h}(p)_{\l}q]+ [p_{\l}\N_{h}(q)]- \N_{h}[p_{\l}q])_{\l+\m}\b(m)
 \\&- l(\N_{h}(\a(q)))_{\m}(l(\N_{h}(p))_{\l}m+ l(p)_{\l}\N_{\M}(m)- \N_{\M} \circ l(p)_{\l}m)\\&+ l(\a(q))_{\l}\N_{\M}(l(\N_{h}(p))_{\l}m+ l(p)_{\l}\N_{\M}(m)- \N_{\M} \circ l(p)_{\l}m)\\&- \N_{\M} \circ l(\a(q))_{\m}(l(\N_{h}(p))_{\l}m+ l(p)_{\l}\N_{\M}(m)- \N_{\M} \circ l(p)_{\l}m)
 \\=& l(\N_{h}(\a(p)))_{\l}(l(\N_{h}(q))_{\m}m)+ l(\N_{h}(\a(p)))_{\l}(l(q)_{\m}\N_{\M}(m))- l(\N_{h}(\a(p)))_{\l}(\N_{\M} \circ l(q)_{\m}m)\\&
 + l(\a(p))_{\l}\N_{\M}(l(\N_{h}(q))_{\m}m)+ l(\a(p))_{\l}\N_{\M}(l(q)_{\m}\N_{\M}(m))-l(\a(p))_{\l}\N_{\M}(\N_{\M} \circ l(q)_{\m}m)\\&
 - \N_{\M} \circ l(\a(p))_{\l}(l(\N_{h}(q))_{\m}m)-\N_{\M} \circ l(\a(p))_{\l}(l(q)_{\m}\N_{\M}(m))+\N_{\M} \circ l(\a(p))_{\l}(\N_{\M} \circ l(q)_{\m}m) \\&
 - l(([\N_{h} (p)_\l \N_{h} (q)]))_{\l+\m} \b(m)\\&
 - l([\N_{h}(p)_{\l}q])_{\l+\m}\N_{\M}(\b(m))- ([p_{\l}\N_{h}(q)])_{\l+\m}\N_{\M}(\b(m))- (\N_{h}[p_{\l}q])_{\l+\m}\N_{\M}(\b(m))\\&
 + \N_{\M} \circ l([\N_{h}(p)_{\l}q])_{\l+\m}\b(m)+ \N_{\M} \circ l([p_{\l}\N_{h}(q)])_{\l+\m}\b(m)- \N_{\M} \circ l(\N_{h}[p_{\l}q])_{\l+\m}\b(m)
 \\&- l(\N_{h}(\a(q)))_{\m}(l(\N_{h}(p))_{\l}m)-l(\N_{h}(\a(q)))_{\m}(l(p)_{\l}\N_{\M}(m))+l(\N_{h}(\a(q)))_{\m}(\N_{\M} \circ l(p)_{\l}m)\\&+ l(\a(q))_{\l}\N_{\M}(l(\N_{h}(p))_{\l}m)+ l(\a(q))_{\l}\N_{\M}(l(p)_{\l}\N_{\M}(m))- l(\a(q))_{\l}\N_{\M}(\N_{\M} \circ l(p)_{\l}m)\\&- \N_{\M} \circ l(\a(q))_{\m}(l(\N_{h}(p))_{\l}m)-\N_{\M} \circ l(\a(q))_{\m}(l(p)_{\l}\N_{\M}(m))+\N_{\M} \circ l(\a(q))_{\m}(\N_{\M} \circ l(p)_{\l}m) 
 \\=& l(\N_{h}(\a(p)))_{\l}(l(\N_{h}(q))_{\m}m) 
- l([\N_{h} (p)_\l \N_{h} (q)])_{\l+\m} \b(m)
- l(\N_{h}(\a(q)))_{\m}(l(\N_{h}(p))_{\l}m)\\&
-\N_{\M}\circ l(\N_{h}(\a(p)))_{\l}(l(q)_{\m}m)
- \N_{\M}\circ l(\a(p))_{\l}(\N_{\M} \circ l(q)_{\m}m)
+ \N_{\M}\circ \N_{\M}\circ l(\a(p))_{\l}(l(q)_{\m}m)
\\&+ l(\N_{h}(\a(p)))_{\l}(l(q)_{\m}\N_{\M}(m))\\&
+ l(\a(p))_{\l}\N_{\M}(l(\N_{h}(q))_{\m}m)
+l(\a(p))_{\l}\N_{\M}(l(q)_{\m}\N_{\M}(m))
-l(\a(p))_{\l}\N_{\M}(\N_{\M} \circ l(q)_{\m}m)\\&
- \N_{\M} \circ l(\a(p))_{\l}(l(\N_{h}(q))_{\m}m)
- \N_{\M} \circ l(\a(p))_{\l}(l(q)_{\m}\N_{\M}(m))
+ \N_{\M} \circ l(\a(p))_{\l}(\N_{\M} \circ l(q)_{\m}m)\\&
+ \N_{\M} \circ l([\N_{h}(p)_{\l}q])_{\l+\m}\b(m)
+ \N_{\M} \circ l([p_{\l}\N_{h}(q)])_{\l+\m}\b(m)
 - \N_{\M} \circ l(\N_{h}[p_{\l}q])_{\l+\m}\b(m)
\\&- l([\N_{h}(p)_{\l}q])_{\l+\m}\N_{\M}(\b(m))
-l([p_{\l}\N_{h}(q)])_{\l+\m}\N_{\M}(\b(m))\\&
+\N_{\M}\circ l([p_{\l}q])_{\l+\m}\N_{\M}(\b(m))
+\N_{\M} \circ l(\N_{h}[p_{\l}q])_{\l+\m}(\b(m))
- \N_{\M}\circ \N_{\M}\circ l([p_{\l}q])_{\l+\m}(\b(m))\\&
+\N_{\M}\circ l(\N_{h}(\a(q)))_{\m}(l(p)_{\l}m)
+\N_{\M}\circ l((\a(q)))_{\m}(\N_{\M} \circ l(p)_{\l}m)
-\N_{\M}\circ l((\a(q)))_{\m}(l(p)_{\l}m)\\&
-l(\N_{h}(\a(q)))_{\m}(l(p)_{\l}\N_{\M}(m))\\&
- l(\a(q))_{\m}\N_{\M}(l(\N_{h}(p))_{\l}m)
+ l(\a(q))_{\m}\N_{\M}(l(p)_{\l}\N_{\M}(m))
- l(\a(q))_{\m}\N_{\M}(\N_{\M} \circ l(p)_{\l}m)\\&
+ \N_{\M} \circ l(\a(q))_{\m}(l(\N_{h}(p))_{\l}m)
-\N_{\M} \circ l(\a(q))_{\m}(l(p)_{\l}\N_{\M}(m))
+\N_{\M} \circ l(\a(q))_{\m}(\N_{\M} \circ l(p)_{\l}m)
\\=&-\N_{\M}\circ l(\N_{h}(\a(p)))_{\l}(l(q)_{\m}m)+ \N_{\M} \circ l([\N_{h}(p)_{\l}q])_{\l+\m}\b(m)+ \N_{\M} \circ l(\a(q))_{\m}(l(\N_{h}(p))_{\l}m)\\&
- \N_{\M}\circ l(\a(p))_{\l}(\N_{\M} \circ l(q)_{\m}m)
+ \N_{\M} \circ l(\a(p))_{\l}(\N_{\M} \circ l(q)_{\m}m)\\&
+ \N_{\M}\circ \N_{\M}\circ l(\a(p))_{\l}(l(q)_{\m}m)
- \N_{\M}\circ \N_{\M}\circ l([p_{\l}q])_{\l+\m}(\b(m))-\N_{\M}\circ\N_{\M}\circ l((\a(q)))_{\m}(l(p)_{\l}m)\\&
 - \N_{\M} \circ l(\N_{h}[p_{\l}q])_{\l+\m}\b(m)
 +\N_{\M} \circ l(\N_{h}[p_{\l}q])_{\l+\m}(\b(m))\\&
+\N_{\M} \circ l(\a(q))_{\m}(\N_{\M} \circ l(p)_{\l}m)
 +\N_{\M} \circ l((\a(q)))_{\m}(\N_{\M} \circ l(p)_{\l}m)
\\&+ l(\N_{h}(\a(p)))_{\l}(l(q)_{\m}\N_{\M}(m))-l([\N_{h}(p)_{\l}q])_{\l+\m}\N_{\M}(\b(m))\\&
- l(\a(q))_{\m}\N_{\M}(l(\N_{h}(p))_{\l}m)
+ l(\a(q))_{\m}\N_{\M}(l(p)_{\l}\N_{\M}(m))
- l(\a(q))_{\m}\N_{\M}(\N_{\M} \circ l(p)_{\l}m)\\&
 +\N_{\M} \circ l(\a(q))_{\m}(l(p)_{\l}\N_{\M}(m))
 - \N_{\M} \circ l(\a(p))_{\l}(l(q)_{\m}\N_{\M}(m))
 +\N_{\M}\circ l([p_{\l}q])_{\l+\m}\N_{\M}(\b(m))\\&
 + \N_{\M} \circ l([p_{\l}\N_{h}(q)])_{\l+\m}\b(m)
 - \N_{\M} \circ l(\a(p))_{\l}(l(\N_{h}(q))_{\m}m)
 +\N_{\M}\circ l(\N_{h}(\a(q)))_{\m}(l(p)_{\l}m)\\&
-l([p_{\l}\N_{h}(q)])_{\l+\m}\N_{\M}(\b(m))-l(\N_{h}(\a(q)))_{\m}(l(p)_{\l}\N_{\M}(m))\\&
+ l(\a(p))_{\l}\N_{\M}(l(\N_{h}(q))_{\m}m)
+l(\a(p))_{\l}\N_{\M}(l(q)_{\m}\N_{\M}(m))
-l(\a(p))_{\l}\N_{\M}(\N_{\M} \circ l(q)_{\m}m)\\=&0.
\end{align*}
\end{enumerate}Similarly, we can show that \begin{align*}
 r'(\b(m)) _\l ([p_\m q])&= r'(r'(m)_\l p)_{\l+\m} \a(q) + l'(\a(p))_\m (r'(m)_\l q),\\
r'(\b(m))_\l([p_\m q])&= -r'(l'(p)_\m m)_{\l+\m} \a(q)+ l'(\a(p))_\m (r'(m)_{\l} q). 
 \end{align*} 	
Hence, $(\M, l' , r',\b )$ is a representation of Hom-Leibniz conformal algebra $(\L_{\N_h}, [\cdot_\l \cdot]_{\N_h}, \a)$. Next, we see that 
\begin{align*}
 &l'(\N_{h} (p))_\l \N_{\M}(m)\\=&l(\N_{h}\N_{h}(p))_{\l} \N_{\M}(m)+ l(\N_{h}(p))_{\l}\N_{\M}\N_{\M}(m)- \N_{\M} \circ l(\N_{h}(p))_{\l}\N_{\M}(m)\\=& \N_{\M}(l(\N_{h}\N_{h}(p))_{\l}m+ l(\N_{h}(p))_{\l}\N_{\M}(m)- \N_{\M} \circ l(\N_{h}(p))_{\l}m) \\&+ \N_{\M}(l(\N_{h}(p))_{\l}\N_{\M}(m)+ l(p)_{\l}\N_{\M}\N_{\M}(m)- \N_{\M} \circ l(p)_{\l}\N_{\M}(m))\\&
 - \N_{\M} \circ \N_{\M}(l(\N_{h}(p))_{\l}m+ l(p)_{\l}\N_{\M}(m)- \N_{\M} \circ l(p)_{\l}m)\\
 =&\N_{\M} \circ (l(\N_{h}\N_{h}(p))_{\l}m+ l(\N_{h}(p))_{\l}\N_{\M}(m)- \N_{\M} \circ l(\N_{h}(p))_{\l}m) \\&
 +\N_{\M} \circ (l(\N_{h}(p))_{\l}\N_{\M}(m)+ l(p)_{\l}\N_{\M}\N_{\M}(m)- \N_{\M} \circ l(p)_{\l}\N_{\M}(m))\\& 
 -\N_{\M}\circ\N_{\M} \circ (l(\N_{h}(p))_{\l}m+ l(p)_{\l}\N_{\M}(m)- \N_{\M} \circ l(p)_{\l}m)
 \\=&\N_{\M} \circ l'(\N_{h} (p))_\l (m) +\N_{\M} \circ l'( p)_\l \N_{\M}(m) -\N_{\M}\circ\N_{\M} \circ l'(p)_\l (m)\\=&\N_{\M} \circ (l'(\N_{h} (p))_\l (m) +l'( p)_\l \N_{\M}(m) -\N_{\M} \circ l'(p)_\l (m)).
 \end{align*} Similarly, we can show that \begin{align*}
 r'(\N_{\M}(m))_{\l} \N_{h}(p)&= \N_{\M}(r'(\N_{\M}(m))_{\l}p+ r'(m)_{\l}\N_{h}(p)- \N_{\M} \circ r'(m)_{\l}p).
 \end{align*}Thus, $(\M, l' , r',\b ,\N_{\M})$ is a representation of Hom-Nijenhuis Leibniz conformal algebra $(\L_{\N_h}, [\cdot_\l \cdot]_{\N_h}, \a)$.
\end{proof}
\section{ Cohomology of Hom-Nijenhuis-Leibniz conformal algebra} To define the cohomology of a Hom-Nijenhuis-Leibniz conformal algebra, we first recall the cohomology of a Hom-Leibniz conformal algebra. It has not been defined previously, we will use the technique of the paper \cite{AW}. 
Let us consider that $(\L,[\cdot_\l \cdot],\a)$ is a Hom-Leibniz conformal algebra, which has coefficients in the representation $(l,r,\M,\b)$. The group of $nth$-cochain $C^{n}_{HomL}(\L,M)$ contains the map $f_{\l_1,\cdots,\l_{n-1}}: \wedge^{\otimes n} \L\to \mathbb{C}[\l_1,\cdots,\l_{n-1}] \otimes \M$ defined by $p_1\otimes p_2\otimes\cdots \otimes p_n \mapsto f_{\l_1,\l_2,\cdots,\l_{n-1}}(p_1,p_2,\cdots,p_n).$ These maps satisfy the following conformal sesqui-linearity condition:
	\begin{equation}\label{eq9}
	\begin{aligned}
	f_{\l_1, \l_2,\cdots, \l_{n-1}}&(p_1,p_2,\cdots,\p(p_i),\cdots,p_n)\\=&\begin{cases}-\l_{i} f_{\l_1,\l_2,\cdots,\l_{n-1}}(p_1, p_2, \cdots , p_n), &i= 1, \cdots, n-1, \\
	 (\p+\l_1+\l_2+\cdots+\l_{n-1}) f_{\l_1,\l_2,\cdots,\l_{n-1}}(p_1, p_2, \cdots , p_n),&i=n. \end{cases}
	\end{aligned}\end{equation} such that $f\circ \a^{\otimes n}= \b\circ f$. The coboundary map $\delta_{Hom}: C^{n}_{Hom}(\L, \M)\to C^{n+1}_{Hom}(\L, \M)$ can be defined by 
\begin{equation}\begin{aligned}\label{coboundarymap}&(\delta_{HomL} f)_{\l_1,\cdots,\l_{n}}(p_1,\cdots, p_{n+1})\\=& \sum_{i=1}^{n}(-1)^{i+1}l(\a^{n-1}(p_i))_{\l_i}f_{\l_1, \cdots, \hat{\l_{i}}, \cdots, \l_{n}}(p_1, \cdots, \hat{p_i}, \cdots, p_{n+1})\\&
+ (-1)^{n+1}r(f_{\l_1, \cdots, \l_{n-1}}(p_1, \cdots, p_{n}))_{\l_1+ \cdots+ \l_n}\a^{n-1}(p_{n+1})\\&
+ \sum_{i<j}(-1)^{i+ j} f_{\l_i+ \l_j, \l_1, \cdots, \hat{\l_{i}}, \cdots, \hat{\l_{j}}, \cdots, \l_{n}}([{p_{i}}_{\l_i} p_j], \a(p_1), \cdots, \hat{\a(p_i)}, \cdots, \hat{\a(p_j)}, \cdots,\a(p_{n+1})), \end{aligned}\end{equation}
 \noindent where $f\in C^{n}_{HomL}(\L, \M)$ and $p_1, p_2, \cdots, p_{n},p_{n+1} \in \L$. With coefficients in the representation $(l,r,\M,\b)$, the corresponding cohomology groups are known as the cohomology of the Hom-Leibniz conformal algebra $(\L,[\cdot_\l\cdot],\a)$.
 If we assume that $l(p)_\l m=[p_\l m]$ and $r(m)_\l p=[m_\l p]$ then above equations turns to be \begin{equation}\begin{aligned}&(\delta_{HomL} f)_{\l_1,\cdots,\l_{n}}(p_1,\cdots, p_{n+1})\\=& \sum_{i=1}^{n}(-1)^{i+1}[{\a^{n-1}(p_i)}_{\l_i}f_{\l_1,\cdots,\hat{\l_{i}},\cdots, \l_{n}}(p_1,\cdots,\hat{p_i},\cdots, p_{n+1})]\\&
+ (-1)^{n+1}[{f_{\l_1,\cdots, \l_{n-1}}(p_1,\cdots, p_{n})}_{\l_1+\cdots+\l_n} \a^{n-1}(p_{n+1})]\\&
+ \sum_{i<j}(-1)^{i+j} f_{\l_i+ \l_j ,\l_1,\cdots, \hat{\l_{i}},\cdots, \hat{\l_{j}},\cdots, \l_{n}}([{p_{i}}_{\l_i} p_j], \a(p_1), \cdots, \hat{\a(p_i)},\cdots, \hat{\a(p_j)},\cdots,\a(p_{n+1})), \end{aligned}\end{equation}
\begin{thm}
	$\delta_{HomL}^{2}\circ\delta_{HomL}^{1} = 0$.
\end{thm}
\begin{proof}Consider that for $ n=1$, we have 
	\begin{align*}
	(\delta_{HomL}^1f)(p_1, p_2)&=l{(p_1)}_\l f(p_2) + r(f(p_1))_\l p_2 - f([{p_1}_\l p_2])\\=& [{p_1}_\l f(p_2)] + [f(p_1)_\l p_2 ]- f([{p_1}_\l p_2]).
	\end{align*}  and 
\begin{equation}\label{eqdeltahom2}\begin{aligned}
(\delta_{HomL}^{2} f)_{\l_1,\l_2}(p_1, p_2, p_3) = & l(\a(p_1))_{\l_1}f_{\l_2}(p_2, p_3)- l(\a(p_2))_{\l_2}f_{\l_1}(p_1, p_3)\\
& - r(f_{\l_1}(p_1, p_2))_{\l_1+\l_2} \a(p_3) - f_{\l_1+\l_2}([{p_1}_{\l_1}{p_2}], \a(p_3))\\
& + f_{\l_1+\l_2}([{p_1}_{\l_1} p_3],\a(p_2)) - f_{\l_1+\l_2}([{p_2}_{\l_2}p_3], \a(p_1))\\= &[(\a(p_1))_{\l_1}f_{\l_2}(p_2, p_3)]- [(\a(p_2))_{\l_2}f_{\l_1}(p_1, p_3)]\\
& - [(f_{\l_1}(p_1, p_2))_{\l_1+\l_2} \a(p_3)] - f_{\l_1+\l_2}([{p_1}_{\l_1}{p_2}], \a(p_3))\\
& + f_{\l_1+\l_2}([{p_1}_{\l_1} p_3],\a(p_2)) - f_{\l_1+\l_2}([{p_2}_{\l_2}p_3], \a(p_1)).
\end{aligned}\end{equation}So we have 
\begin{align*}&(\delta_{HomL}^{2}\circ\delta_{HomL}^{1} f)_{\l_1,\l_2}(p_1, p_2, p_3)\\ &=[(\a(p_1))_{\l_1}\delta_{HomL}^1f_{\l_2}(p_2, p_3)]- [(\a(p_2))_{\l_2}\delta_{HomL}^1f_{\l_1}(p_1, p_3)]\\
& - [(\delta_{HomL}^1f_{\l_1}(p_1, p_2))_{\l_1+\l_2} \a(p_3)] - \delta_{HomL}^1f_{\l_1+\l_2}([{p_1}_{\l_1}{p_2}], \a(p_3))\\
& + \delta_{HomL}^1f_{\l_1+\l_2}([{p_1}_{\l_1} p_3],\a(p_2)) - \delta_{HomL}^1f_{\l_1+\l_2}([{p_2}_{\l_2}p_3], \a(p_1))\\=&
[(\a(p_1))_{\l_1}\{ [{p_2}_{\l_2} f(p_3)] + [f(p_2)_{\l_2} p_3 ]- f([{p_2}_{\l_2} p_3])\}]\\&- [(\a(p_2))_{\l_2}\{ [{p_1}_{\l_1} f(p_3)] + [f(p_1)_{\l_1} p_3 ]- f([{p_1}_{\l_1} p_3])\}]\\& - [(\{ [{p_1}_{\l_1} f(p_2)] + [f(p_1)_{\l_1} p_2 ]- f([{p_1}_{\l_1} p_2])\})_{\l_1+\l_2} \a(p_3)] 
\\&- \{ [{[{p_1}_{\l_1}{p_2}]}_{\l_1+\l_2} f(\a(p_3))] + [f([{p_1}_{\l_1}{p_2}])_{\l_1+\l_2} \a(p_3) ]- f([{[{p_1}_{\l_1}{p_2}]}_{\l_1+\l_2} \a(p_3)])\}\\& +\{ [{[{p_1}_{\l_1} p_3]}_{\l_1+\l_2} f(\a(p_2))] + [f([{p_1}_{\l_1} p_3])_{\l_1+\l_2} \a(p_2) ]- f([{[{p_1}_{\l_1} p_3]}_{\l_1+\l_2} \a(p_2)])\}\\&
-\{ [{[{p_2}_{\l_2}p_3]}_{\l_1+\l_2} f(\a(p_1))] + [f([{p_2}_{\l_2}p_3])_{\l_1+\l_2} \a(p_1) ]- f([{[{p_2}_{\l_2}p_3]}_{\l_1+\l_2} \a(p_1)])\}\\=&
[(\a(p_1))_{\l_1}[{p_2}_{\l_2} f(p_3)]] + [(\a(p_1))_{\l_1}[f(p_2)_{\l_2} p_3 ]]- [(\a(p_1))_{\l_1}f([{p_2}_{\l_2} p_3])]\\&- [(\a(p_2))_{\l_2}[{p_1}_{\l_1} f(p_3)]] - [(\a(p_2))_{\l_2}[f(p_1)_{\l_1} p_3 ]]+ [(\a(p_2))_{\l_2}f([{p_1}_{\l_1} p_3])]\\& - [([{p_1}_{\l_1} f(p_2)])_{\l_1+\l_2} \a(p_3)] -[([f(p_1)_{\l_1} p_2 ])_{\l_1+\l_2} \a(p_3)]+ [(f([{p_1}_{\l_1} p_2]))_{\l_1+\l_2} \a(p_3)] 
\\&- \{[{[{p_1}_{\l_1}{p_2}]}_{\l_1+\l_2} f(\a(p_3))] + [f([{p_1}_{\l_1}{p_2}])_{\l_1+\l_2} \a(p_3) ]- f([{[{p_1}_{\l_1}{p_2}]}_{\l_1+\l_2} \a(p_3)])\}\\& +\{ [{[{p_1}_{\l_1} p_3]}_{\l_1+\l_2} f(\a(p_2))] + [f([{p_1}_{\l_1} p_3])_{\l_1+\l_2} \a(p_2) ]- f([{[{p_1}_{\l_1} p_3]}_{\l_1+\l_2} \a(p_2)])\}\\&
-\{ [{[{p_2}_{\l_2}p_3]}_{\l_1+\l_2} f(\a(p_1))] + [f([{p_2}_{\l_2}p_3])_{\l_1+\l_2} \a(p_1) ]- f([{[{p_2}_{\l_2}p_3]}_{\l_1+\l_2} \a(p_1)])\} \\=&0.\end{align*}
Thus, our conclusion holds.

\end{proof}Let $\L_{\N_h}$ be a Hom-Nijenhuis Leibniz conformal algebra with representation $(\M, l , r , \N_{\M},\b )$. Now by Proposition \eqref{prop2.9} and Proposition \eqref{prop2.12}, we obtain a new Hom-Nijenhuis Leibniz conformal algebra $(\L_{\N_h}, [\cdot_\l \cdot]_{\N_h},\a)$ with representation
 $(\M, l' , r' , \b, \N_{\M} )$ induced by the Nijenhuis operator $\N_h$.
 For $n\geq 0$, the coboundary map $\partial_{HN} : C^{n}_{HN}(\L, \M )\to C^{n+1}_{HN}(\L,\M)$ can be defined by 
\begin{align*}&(\partial_{HN} f)_{\l_1,\cdots,\l_{n}}(p_1,\cdots, p_{n+1})\\=& \sum_{i=1}^{n}(-1)^{i+1}l'(\a^{n-1}(p_i))_{\l_i}f_{\l_1,\cdots,\hat{\l_{i}},\cdots, \l_{n}}(p_1,\cdots,\hat{p_i},\cdots, p_{n+1})\\&
+ (-1)^{n+1}r'{(f_{\l_1,\cdots, \l_{n-1}}(p_1,\cdots, p_{n}))}_{\l_1+\l_2+\cdots+\l_n} \a^{n-1}(p_{n+1})\\&
+ \sum_{i<j}(-1)^{i+j} f_{\l_i+ \l_j ,\l_1,\cdots, \hat{\l_{i}},\cdots, \hat{\l_{j}},\cdots, \l_{n}}([{p_{i}}_{\l_i} p_j]_{\N_h}, \a(p_1), \cdots, \hat{\a(p_i)},\cdots, \hat{\a(p_j)},\cdots,\a(p_{n+1})),
\\=& \sum_{i=1}^{n}
(-1)^{i+1} l(\N_{h}(\a^{n-1}(p_i)))_{\l_i}f_{\l_1,\cdots,\hat{\l_{i}},\cdots, \l_{n}}(p_1,\cdots,\hat{p_i},\cdots, p_{n+1})\\&+ \sum_{i=1}^{n}(-1)^{i+1}l(\a^{n-1}(p_i))_{\l_i}\N_{\M}(f_{\l_1,\cdots,\hat{\l_{i}},\cdots, \l_{n}}(p_1,\cdots,\hat{p_i},\cdots, p_{n+1}))\\&- \sum_{i=1}^{n}(-1)^{i+1} \N_{\M} \circ l(\a^{n-1}(p_i))_{\l_i}f_{\l_1,\cdots,\hat{\l_{i}},\cdots, \l_{n}}(p_1,\cdots,\hat{p_i},\cdots, p_{n+1})\\&
+(-1)^{n+1}r(\N_{\M}((f_{\l_1,\cdots, \l_{n-1}}(p_1,\cdots, p_{n}))))_{\l_1+\l_2+\cdots+\l_n} \a^{n-1}(p_{n+1})\\& +(-1)^{n+1} r((f_{\l_1,\cdots, \l_{n-1}}(p_1,\cdots, p_{n})))_{\l_1+\l_2+\cdots+\l_n}\N_{h}(\a^{n-1}(p_{n+1}))\\&- (-1)^{n+1}\N_{\M} \circ r((f_{\l_1,\cdots, \l_{n-1}}(p_1,\cdots, p_{n})))_{\l_1+\l_2+\cdots+\l_n}\a^{n-1}(p_{n+1})\\&
+ \sum_{i<j}(-1)^{i+j} f_{\l_i+ \l_j ,\l_1,\cdots, \hat{\l_{i}},\cdots, \hat{\l_{j}},\cdots, \l_{n}}([{\N_{h}(p_i)}_{\l_i}p_j], \a(p_1), \cdots, \hat{\a(p_i)},\cdots, \hat{\a(p_j)},\cdots,\a(p_{n+1}))\\&
+ \sum_{i<j}(-1)^{i+j} f_{\l_i+ \l_j ,\l_1,\cdots, \hat{\l_{i}},\cdots, \hat{\l_{j}},\cdots, \l_{n}}([{p_i}_{\l_i}\N_{h}(p_j)], \a(p_1), \cdots, \hat{\a(p_i)},\cdots, \hat{\a(p_j)},\cdots,\a(p_{n+1}))\\&- \sum_{i<j}(-1)^{i+j} f_{\l_i+ \l_j ,\l_1,\cdots, \hat{\l_{i}},\cdots, \hat{\l_{j}},\cdots, \l_{n}}(\N_{h}[{p_i}_{\l_i}{p_j}], \a(p_1), \cdots, \hat{\a(p_i)},\cdots, \hat{\a(p_j)},\cdots,\a(p_{n+1})),
 \end{align*} where $f\in C^{n}_{HomL}(\L, \M )$ and $p_1,p_2,\cdots,p_{n+1}\in \L$.
 The above map satisfies the condition $\partial_{HN}^n\circ \partial_{HN}^{n+1}=0$, as it is a coboundary map for the induced Hom-Leibniz conformal algebra $(\L, [-_\l-]_{\N_h},\a)$. Therefore $ (C^{n}_{HN}(\L, \M ),\partial_{HN})$ is a cochain complex. This cochain complex is called the cochain complex of a Hom-Nijenhuis operator, with coefficients in the representation $(\M,l,r,\b,\N_{\M})$.\\ 
 Let us now define a map  $\phi^n:(C^{n}_{HL}(\L, \M ),\delta) \to (C^{n}_{HN}(\L, \M ),\partial)$, for $n\geq 1$ by 
	\begin{align*}&\phi^n(f)_{\l_1,\l_2,\cdots,\l_{n-1}}(p_1, p_2, \cdots, p_n) \\=& f_{\l_1,\l_2,\cdots,\l_{n-1}}(\N_{h}(p_1), \N_{h} (p_2), \cdots, \N_{h} (p_n)) - (\N_{\M} \circ f)_{\l_1,\l_2,\cdots,\l_{n-1}}(p_1, \N_{h} (p_2), \cdots, \N_{h} (p_n))\\& - (\N_{\M} \circ f)_{\l_1,\l_2,\cdots,\l_{n-1}}(\N_{h} (p_1), p_2, \N_{h} (p_3), \cdots, \N_{h} (p_n)) - \cdots\\& - (\N_{\M} \circ f)_{\l_1,\l_2,\cdots,\l_{n-1}}(\N_{h} (p_1), \N_{h} (p_2), \cdots, \N_{h} (p_{n-1}), p_n) + \N_{\M}^2(f_{\l_1,\l_2,\cdots,\l_{n-1}}(p_1, p_2, \cdots, p_n)).
	\end{align*}
	For $n= 2$, we get 
	\begin{align*}
	\phi^2(f)_{\l_1}(p_1, p_2) = f_{\l_1}(\N_{h} (p_1), \N_{h} (p_2)) - (\N_{\M} \circ f)_{\l_1}(p_1, \N_{h} (p_2))
	- (\N_{\M} \circ f)_{\l_1}(\N_{h} (p_1), p_2) + \N_{\M}^2(f_{\l_1}(p_1, p_2)). \end{align*}
 \begin{lem}
 	For $f \in C^{n}_{HomL}(\L, \M )$ and $p_1, \cdots , p_{n+1} \in \L$, we have $$\phi^{n+1}(\delta_{HomL}^n(f))(p_1, p_2, p_3, \cdots , p_{n+1}) = \partial^n_{HN}(\phi^n(f))(p_1, p_2, p_3, \cdots , p_{n+1}).$$ 
 \end{lem}
\begin{proof} Firstly, we consider that
\begin{align*}&\phi^{n+1}(\delta_{HomL}^n(f))(p_1, p_2, p_3, \cdots , p_{n+1})
\\=& \delta_{HomL}^n f_{\l_1,\l_2,\cdots, \l_{n}}(\N_{h} (p_1), \N_{h} (p_2), \cdots, \N_{h} (p_n), \N_{h} (p_{n+1}))
\\&- (\N_{\M} \circ \delta_{HomL}^nf)_{\l_1,\l_2,\cdots,\l_{n}}(p_1, \N_{h} (p_2), \cdots, \N_{h} (p_n), \N_{h} (p_{n+1}))
\\& 
- (\N_{\M} \circ \delta_{HomL}^nf)_{\l_1,\l_2,\cdots,\l_{n}}(\N_{h} (p_1), p_2, \N_{h} (p_3), \cdots, \N_{h} (p_n), \N_{h} (p_{n+1}))\\&- \cdots 
 - (\N_{\M} \circ \delta_{HomL}^nf)_{\l_1,\l_2,\cdots,\l_{n}}(\N_{h} (p_1), \N_{h} (p_2), \cdots, \N_{h} (p_{n-1}), \N_{h} (p_{n}), p_{n+1}) 
\\&+ \N_{\M}^2(\delta_{HomL}^nf_{\l_1,\l_2,\cdots,\l_{n}}(p_1, p_2, \cdots, p_n))\\=& \sum_{i=1}^{n}(-1)^{i+1}[(\a^{n-1}(\N_{h} (p_i)))_{\l_i}f_{\l_2,\cdots,\hat{\l_{i}},\cdots, \l_{n+1}}(\N_{h} (p_1),\cdots,\hat{\N_{h} (p_i)},\cdots, \N_{h} (p_{n+1}))]
\\& + (-1)^{n+1}[f_{\l_1,\cdots, \l_{n-1}}(\N_{h} (p_1),\cdots, \N_{h} (p_{n}))_{\l_1+\cdots+\l_n} \a^{n-1}(\N_{h} (p_{n+1}))]
\\& + \sum_{i<j}(-1)^{i+j} f_{\l_i+ \l_j ,\l_1,\cdots, \hat{\l_{i}},\cdots, \hat{\l_{j}},\cdots, \l_{n}}
\\&([{\N_{h} (p_i)}_{\l_i} \N_{h} (p_j)], \a(\N_{h} (p_1)), \cdots, \hat{\a(\N_{h} (p_i))},\cdots, \hat{\a(\N_{h} (p_j))},\cdots,\a(\N_{h} (p_{n+1})))
\\&- \N_{\M}(\left[ \a^{n-1}{(p_1)}_{\l_{1}} f_{\l_2,\cdots, \l_{n}}(\N_{h} (p_2), \N_{h} (p_3), \cdots, \N_{h} (p_{n+1})) \right]
\\& - [\a^{n-1}{(\N_{h} (p_2))}_{\l_2 }f_{\l_1,\cdots,\hat{\l_{2}},\cdots, \l_{n}}(p_1, \N_{h} (p_3), \cdots, \N_{h} (p_{n+1}))] \\&+ [\a^{n-1}(\N_{h} {p_3})_{\l_3 }f_{\l_1,\cdots,\hat{\l_{3}},\cdots, \l_{n}}(\N_{h} (p_1), p_2, \N_{h} (p_4), \cdots, \N_{h} (p_{n+1}))] \\&- \cdots
+ (-1)^{n+1}[\a^{n-1}(\N_{h} {p_n})_{\l_n} f_{\l_2,\cdots,\hat{\l_{n}},\cdots, \l_{n}}(p_1, \N_{h} (p_2), \cdots, \N_{h} (p_{n-1}), \N_{h} (p_{n+1}))]
\\& + (-1)^{n+1}[f_{\l_1,\cdots, \l_{n-1}}(p_1, \N_{h} (p_2), \N_{h} (p_3), \cdots, \N_{h} (p_n))_{\l_1+\cdots+\l_{n}} \a^{n-1}(\N_{h} (p_{n+1}))] 
\\&- \sum_{j=2}^{n+1}(-1)^{1+j} f_{\l_1+\l_j,\hat{\l_1},\l_2,\cdots,\hat{\l_j},\cdots, \l_{n}}
\\&([{p_1}_{\l_1} \N_{h} (p_j)], \a(\N_{h} (p_2)), \a(\N_{h} (p_3)), \cdots, \a(\N_{h} (p_{j-1})), \hat{\a(\N_{h} (p_j))}, \cdots, \a(\N_{h} (p_{n+1})))
\\&+ \sum_{2\leq i < j \leq n+1,i\neq1}(-1)^i f_{\l_i+\l_j,\l_1,\cdots,\hat{\l_{i}},\cdots,\hat{\l_j},\cdots, \l_{n}}
([{\N_{h} (p_i)}_{\l_i} \N_{h} (p_j)],\a(p_1), \a(\N_{h} (p_2)), \cdots, \\&\a(\hat{\N_{h} (p_i)}), \cdots, \hat{\a(\N_{h} (p_j))}, \a(\N_{h} (p_{j+1})), \cdots, \a(\N_{h} (p_{n+1}))))
\\&
\\& -\N_{\M}(\left[ \a^{n-1}{(\N_{h} (p_1))}_{\l_1} f_{\l_2,\cdots, \l_{n}}(p_2, \N_{h} (p_3), \cdots, \N_{h} (p_{n+1})) \right]
\\& - [\a^{n-1}{(p_2)}_{\l_2} f_{\l_1,\cdots,\hat{\l_{2}},\cdots, \l_{n}}(\N_{h} (p_1), \hat{p_2}, \N_{h} (p_3), \cdots, \N_{h} (p_{n+1}))]
\\&+ [\a^{n-1}{(\N_{h} (p_3))}_{\l_3} f_{\l_1,\cdots,\hat{\l_{3}},\cdots, \l_{n}}(\N_{h} (p_1), p_2,\hat{\N_{h} (p_3)}, \N_{h} (p_4), \cdots, \N_{h} (p_{n+1}))] \\&- \cdots 
+ (-1)^{n+1}[\a^{n-1}{(\N_{h} (p_n))}_{\l_n} f_{\l_1,\cdots,\hat{\l_{n}},\cdots, \l_{n}}(\N_{h} (p_1), p_2, \cdots, \N_{h} (p_{n-1}), \N_{h} (p_{n+1}))]
\\&+ (-1)^{n+1}[f_{\l_1,\cdots, \l_{n-1}}(\N_{h} (p_1), p_2, \N_{h} (p_3), \cdots, \N_{h} (p_n))_{\l_1+\cdots+ \l_{n}} \a^{n-1}(\N_{h} (p_{n+1}))]
\\&+ \sum_{j=3}^{n+1} f_{\l_2+\l_j,\l_1,\l_3,\cdots,\hat{\l_{j}},\cdots, \l_{n}}([{p_2}_{\l_2} \N_{h} (p_j)],\a(\N_{h} (p_1)),\hat{\a(\N_{h} (p_2))} ,\\& \a(\N_{h} (p_3)), \cdots, \a(\N_{h} (p_{j-1})),\hat{\a(\N_{h} (p_j))} \a(\N_{h} (p_{j+1})), \cdots, \a(\N_{h} (p_{n+1})))
\\&+ \sum_{1\leq i<j\leq n+1, i\neq2}^{n+1}(-1)^i f_{\l_i+\l_j, \l_1,\cdots,\hat{\l_{i}},\cdots,\hat{\l_j},\cdots, \l_{n}}([{\N_{h} (p_i)}_{\l_i }\N_{h} (p_j)], \a(\N_{h} (p_1)), \\&\a(p_2), \a(\N_{h} (p_3)), \cdots, \hat{\a(\N_{h} (p_i)}), \cdots, \hat{\a(\N_{h} (p_j))}, \a(\N_{h} (p_{j+1})), \cdots, \a(\N_{h} (p_{n+1})))) \\&- \cdots 
+\N_{\M}^2 (\sum_{i=1}^{n}(-1)^{i+1}[(\a^{n-1}(p_i))_{\l_i}f_{\l_1,\cdots,\hat{\l_{i}},\cdots, \l_{n}}(p_1,\cdots,\hat{p_i},\cdots, p_{n+1})]
\\&+ (-1)^{n+1}[f_{\l_1,\cdots, \l_{n-1}}(p_1,\cdots, p_{n})_{\l_1+\cdots+\l_n} \a^{n-1}(p_{n+1})]
\\&+ \sum_{i<j}(-1)^{i+j} f_{\l_i+ \l_j ,\l_1,\cdots, \hat{\l_{i}},\cdots, \hat{\l_{j}},\cdots, \l_{n}}([{p_{i}}_{\l_i} p_j], \a(p_1), \cdots, \hat{\a(p_i)},\cdots, \hat{\a(p_j)},\cdots,\a(p_{n+1}))), \end{align*}Again 
	\begin{align*}
& \partial_{HN}^n(\phi^n(f))_{\l_1,\cdots, \l_{n}}(p_1, p_2, p_3, \cdots , p_{n+1})\\
	=&\sum_{i=1}^{n}(-1)^{i+1} l(\N_{h}(\a^{n-1}(p_i)))_{\l_i}\phi^nf_{\l_1,\cdots,\hat{\l_{i}},\cdots, \l_{n}}(p_1,\cdots,\hat{p_i},\cdots, p_{n+1})\\&+ \sum_{i=1}^{n}(-1)^{i+1}l(\a^{n-1}(p_i))_{\l_i}\N_{\M}(\phi^nf_{\l_1,\cdots,\hat{\l_{i}},\cdots, \l_{n}}(p_1,\cdots,\hat{p_i},\cdots, p_{n+1}))\\&- \sum_{i=1}^{n}(-1)^{i+1} \N_{\M} \circ l(\a^{n-1}(p_i))_{\l_i}\phi^nf_{\l_1,\cdots,\hat{\l_{i}},\cdots, \l_{n}}(p_1,\cdots,\hat{p_i},\cdots, p_{n+1})\\&
	+(-1)^{n+1}r(\N_{\M}((\phi^nf_{\l_1,\cdots, \l_{n-1}}(p_1,\cdots, p_{n}))))_{\l_1+\l_2+\cdots+\l_n} \a^{n-1}(p_{n+1})\\& +(-1)^{n+1} r(\phi^nf_{\l_1,\cdots, \l_{n-1}}(p_1,\cdots, p_{n}))_{\l_1+\l_2+\cdots+\l_n}\N_{h}(\a^{n-1}(p_{n+1}))\\&- (-1)^{n+1}\N_{\M}\circ r(\phi^nf_{\l_1,\cdots, \l_{n-1}}(p_1,\cdots, p_{n}))_{\l_1+\l_2+\cdots+\l_n}\a^{n-1}(p_{n+1})\\&
	+ \sum_{i<j}(-1)^{i+j} \phi^nf_{\l_i+ \l_j ,\l_1,\cdots, \hat{\l_{i}},\cdots, \hat{\l_{j}},\cdots, \l_{n}}\\&
    ([{\N_{h}(p_i)}_{\l_i}p_j], \a(p_1), \cdots, \hat{\a(p_i)},\cdots, \hat{\a(p_j)},\cdots,\a(p_{n+1}))\\&
	+ \sum_{i<j}(-1)^{i+j} \phi^nf_{\l_i+ \l_j ,\l_1,\cdots, \hat{\l_{i}},\cdots, \hat{\l_{j}},\cdots, \l_{n}}\\&([{p_i}_{\l_i}\N_{h}(p_j)], \a(p_1), \cdots, \hat{\a(p_i)},\cdots, \hat{\a(p_j)},\cdots,\a(p_{n+1}))
    \\& - \sum_{i<j}(-1)^{i+j} \phi^nf_{\l_i+ \l_j ,\l_1,\cdots, \hat{\l_{i}},\cdots, \hat{\l_{j}},\cdots, \l_{n}}\\&(\N_{h}[{p_i}_{\l_i}{p_j}], \a(p_1), \cdots, \hat{\a(p_i)},\cdots, \hat{\a(p_j)},\cdots,\a(p_{n+1}))\\=&\sum_{i=1}^{n}(-1)^{i+1}[(\N_{h}(\a^{n-1}(p_i)))_{\l_i}\phi^nf_{\l_1,\cdots,\hat{\l_{i}},\cdots, \l_{n}}(p_1,\cdots,\hat{p_i},\cdots, p_{n+1})]
    \\&+ \sum_{i=1}^{n}(-1)^{i+1}[(\a^{n-1}(p_i))_{\l_i}\N_{\M}(\phi^nf_{\l_1,\cdots,\hat{\l_{i}},\cdots, \l_{n}}(p_1,\cdots,\hat{p_i},\cdots, p_{n+1}))]
    \\&- \sum_{i=1}^{n}(-1)^{i+1} \N_{\M} \circ [(\a^{n-1}(p_i))_{\l_i}\phi^nf_{\l_1,\cdots,\hat{\l_{i}},\cdots, \l_{n}}(p_1,\cdots,\hat{p_i},\cdots, p_{n+1})]
    \\& +(-1)^{n+1}[\N_{\M}(\phi^nf_{\l_1,\cdots, \l_{n-1}}(p_1,\cdots, p_{n}))_{\l_1+\l_2+\cdots+\l_n} \a^{n-1}(p_{n+1})]
    \\& +(-1)^{n+1}[\phi^nf_{\l_1,\cdots, \l_{n-1}}(p_1,\cdots, p_{n})_{\l_1+\l_2+\cdots+\l_n}\N_{h}\a^{n-1}(p_{n+1})]
    \\&- (-1)^{n+1}\N_{\M} \circ [\phi^nf_{\l_1,\cdots, \l_{n-1}}(p_1,\cdots, p_{n})_{\l_1+\l_2+\cdots+\l_n}\a^{n-1}(p_{n+1})]
    \\& + \sum_{i<j}(-1)^{i+j} \phi^nf_{\l_i+ \l_j ,\l_1,\cdots, \hat{\l_{i}},\cdots, \hat{\l_{j}},\cdots, \l_{n}}([{\N_{h}(p_i)}_{\l_i}p_j], \a(p_1), \cdots, \hat{\a(p_i)},\cdots, \hat{\a(p_j)},\cdots,\a(p_{n+1}))
    \\&+ \sum_{i<j}(-1)^{i+j} \phi^nf_{\l_i+ \l_j ,\l_1,\cdots, \hat{\l_{i}},\cdots, \hat{\l_{j}},\cdots, \l_{n}}([{p_i}_{\l_i}\N_{h}(p_j)], \a(p_1), \cdots, \hat{\a(p_i)},\cdots, \hat{\a(p_j)},\cdots,\a(p_{n+1}))
    \\&- \sum_{i<j}(-1)^{i+j} \phi^nf_{\l_i+ \l_j ,\l_1,\cdots, \hat{\l_{i}},\cdots, \hat{\l_{j}},\cdots, \l_{n}}(\N_{h}[{p_i}_{\l_i}{p_j}], \a(p_1), \cdots, \hat{\a(p_i)},\cdots, \hat{\a(p_j)},\cdots,\a(p_{n+1}))
	\\=& [\a^{n-1}{\N_{h}(p_1)}_{\l_1} \phi^n(f)_{\l_1,\cdots,\hat{\l_{1}},\cdots, \l_{n}}(p_2, p_3, \cdots , p_{n+1})]
    \\& - [\a^{n-1}{\N_{h}(p_2)}_{\l_2} \phi^n(f)_{\l_1,\cdots,\hat{\l_{2}},\cdots, \l_{n}}(p_1, p_3, \cdots , p_{n+1})]
    \\&+ [\a^{n-1}{\N_{h}(p_3)}_{\l_3}\phi^n(f)_{\l_1,\cdots,\hat{\l_{3}},\cdots, \l_{n}}(p_1, p_2, p_4, \cdots , p_{n+1})]
    \\& - \cdots + (-1)^{n+1}[\a^{n-1}{\N_{h}(p_n)}_{\l_n}\phi^n(f)_{\l_1,\cdots,\hat{\l_{n}},\cdots, \l_{n}}(p_1, p_2, \cdots , p_{n-1}, p_{n+1})]
    \\& + [(\a^{n-1}(p_1))_{\l_1}\N_{\M}(\phi^nf_{\l_1,\cdots,\hat{\l_{1}},\cdots, \l_{n}}(p_1,\cdots,\hat{p_1},\cdots, p_{n+1}))]
    \\&- [(\a^{n-1}(p_2))_{\l_2}\N_{\M}(\phi^nf_{\l_1,\cdots,\hat{\l_{2}},\cdots, \l_{n}}(p_1,\cdots,\hat{p_2},\cdots, p_{n+1}))] +\cdots
    \\&	+ (-1)^{n+1}[(\a^{n-1}(p_n))_{\l_n}\N_{\M}(\phi^nf_{\l_1,\cdots,\hat{\l_{n}},\cdots, \l_{n}}(p_1,\cdots,\hat{p_n},\cdots, p_{n+1}))]
    \\&-\N_{\M} [\a^{n-1}{ (p_1)}_{\l_1} \phi^n(f)_{\l_1,\cdots,\hat{\l_{1}},\cdots, \l_{n}}(p_2, p_3, \cdots , p_{n+1})] +
    \\&\N_{\M} [\a^{n-1}{(p_2)}_{\l_2} \phi^n(f)_{\l_1,\cdots,\hat{\l_{2}},\cdots, \l_{n}}(p_1, p_3, \cdots, p_{n+1})]
    \\&- \N_{\M} [\a^{n-1}{(p_3)}_{\l_3 }\phi^n(f)_{\l_1,\cdots,\hat{\l_{3}},\cdots, \l_{n}}(p_1, p_2, p_4, \cdots, p_{n+1})]
    \\& - \cdots + (-1)^{n+1}\N_{\M}[\a^{n-1}{(p_n)}_{\l_n }\phi^n(f)(p_1, p_2, \cdots , p_{n-1}, p_{n+1})]
    \\&+(-1)^{n+1}[\N_{\M}(\phi^nf_{\l_1,\cdots, \l_{n-1}}(p_1,\cdots, p_{n}))_{\l_1+\l_2+\cdots+\l_n} \a^{n-1}(p_{n+1})]
    \\&+(-1)^{n+1}[\phi^nf_{\l_1,\cdots, \l_{n-1}}(p_1,\cdots, p_{n})_{\l_1+\l_2+\cdots+\l_n}\N_{h}\a^{n-1}(p_{n+1})]
    \\&- (-1)^{n+1}\N_{\M} \circ [\phi^nf_{\l_1,\cdots, \l_{n-1}}(p_1,\cdots, p_{n})_{\l_1+\l_2+\cdots+\l_n}\a^{n-1}(p_{n+1})]
    \\&- \sum_{j=2}^{n+1} \phi^nf_{\l_1+ \l_j ,\l_1,\cdots, \hat{\l_{1}},\cdots, \hat{\l_{j}},\cdots, \l_{n}}
    \\& ([{\N_{h}(p_1)}_{\l_1}p_j]+[{p_1}_{\l_1}\N_{h}(p_j)]-\N_{h}[{p_1}_{\l_1}{p_j}], \a(p_1), \cdots, \hat{\a(p_1)},\cdots, \hat{\a(p_j)},\cdots,\a(p_{n+1}))
    \\&+ \sum_{j=3}^{n+1} \phi^nf_{\l_2+ \l_j ,\l_1,\cdots, \hat{\l_{2}},\cdots, \hat{\l_{j}},\cdots, \l_{n}}
    \\&([{\N_{h}(p_2)}_{\l_2}p_j]+[{p_2}_{\l_2}\N_{h}(p_j)]-\N_{h}[{p_2}_{\l_2}{p_j}], \a(p_1), \cdots, \hat{\a(p_2)},\cdots, \hat{\a(p_j)},\cdots,\a(p_{n+1}))
    \\&- \sum_{j=4}^{n+1} \phi^nf_{\l_3+ \l_j ,\l_1,\cdots, \hat{\l_{3}},\cdots, \hat{\l_{j}},\cdots, \l_{n}}
    \\&([{\N_{h}(p_3)}_{\l_3}p_j]+[{p_3}_{\l_3}\N_{h}(p_j)]-\N_{h}[{p_3}_{\l_2}{p_j}], \a(p_1), \cdots, \hat{\a(p_3)},\cdots, \hat{\a(p_j)},\cdots,\a(p_{n+1}))
    \\&\cdots
    \\&+ \sum_{j=n+1}^{n+1}(-1)^n \phi^nf_{\l_n+ \l_j ,\l_1,\cdots, \hat{\l_{n}},\cdots, \hat{\l_{j}},\cdots, \l_{n}}
    \\&([{\N_{h}(p_n)}_{\l_n}p_j]+[{p_n}_{\l_n}\N_{h}(p_j)]-\N_{h}[{p_n}_{\l_n}{p_j}], \a(p_1), \cdots, \hat{\a(p_n)},\cdots, \hat{\a(p_j)},\cdots,\a(p_{n+1}))
\\=&[\a^{n-1}{\N_{h}(p_1)}_{\l_1}
\\&(f_{\l_2,\l_3,\cdots,\l_{n}}(\N_{h}(p_2), \N_{h}(p_3), \cdots, \N_{h}(p_{n+1})) \\&- (\N_{\M} \circ f)_{\l_2,\l_3,\cdots,\l_{n}}(p_2, \N_{h}(p_3), \cdots, \N_{h}(p_{n+1}))
\\& - (\N_{\M} \circ f)_{\l_2,\l_3,\cdots,\l_{n}}(\N_{h} (p_2), p_3, \N_{h} (p_4), \cdots, \N_{h} (p_{n+1})) \\&- \cdots
- (\N_{\M} \circ f)_{\l_2,\l_3,\cdots,\l_{n}}(\N_{h} (p_2), \N_{h} (p_3), \cdots, \N_{h} (p_{n}), p_{n+1})\\&+ \N_{\M}^2(f_{\l_2,\l_3,\cdots,\l_{n}}(p_2, p_3, \cdots, p_{n+1})))]
\\&- [\a^{n-1}{\N_{h}(p_2)}_{\l_2}
\\& f_{\l_1,\l_3,\cdots,\l_{n}}(\N_{h} (p_1), \N_{h} (p_3), \cdots, \N_{h} (p_{n+1}))\\& - (\N_{\M} \circ f)_{\l_1,\l_3,\cdots,\l_{n}}(p_1, \N_{h} (p_3), \cdots, \N_{h} (p_{n+1}))
\\& - (\N_{\M} \circ f)_{\l_1,\l_3,\cdots,\l_{n}}(\N_{h} (p_1), p_3, \N_{h} (p_4), \cdots, \N_{h} (p_{n+1})) \\&- \cdots - (\N_{\M} \circ f)_{\l_1,\l_3,\cdots,\l_{n}}(\N_{h} (p_1), \N_{h} (p_3), \cdots, \N_{h} (p_{n}), p_{n+1})\\&+ \N_{\M}^2(f_{\l_1,\l_3,\cdots,\l_{n}}(p_1, p_3, \cdots, p_{n+1}))] 
\\&
+ [\a^{n-1}{\N_{h}(p_3)}_{\l_3 }\\&f_{\l_1,\l_2,\l_4,\cdots,\l_{n}}(\N_{h} (p_1), \N_{h} (p_2),\N_{h} (p_4) \cdots, \N_{h} (p_{n+1})) \\&- (\N_{\M} \circ f)_{\l_1,\cdots,\hat{\l_3},\cdots,\l_{n}}(p_1, \N_{h} (p_2),\hat{\N_{h} (p_3)}, \cdots, \N_{h} (p_{n+1}))\\&- (\N_{\M} \circ f)_{\l_1,\cdots,\hat{\l_3},\cdots,\l_{n}}(\N_{h} (p_1), p_2, \hat {\N_{h} (p_3)}, \cdots, \N_{h} (p_n))\\& - \cdots - (\N_{\M} \circ f)_{\l_1,\cdots,\hat{\l_3},\cdots,\l_{n}}(\N_{h} (p_1), \N_{h} (p_2), \cdots, \N_{h} (p_{n}), p_{n+1})\\&+ \N_{\M}^2(f_{\l_1,\cdots,\hat{\l_3},\cdots,\l_{n}}(p_1, p_2,p_4, \cdots, p_{n+1}))]
\\&- \cdots + (-1)^{n+1}[\a^{n-1}{\N (p_n)}_{\l_n}\\& f_{\l_1,\l_2,\cdots,\l_{n-1}}(\N_{h} (p_1), \N_{h} (p_2), \cdots, \N_{h} (p_{n-1}),\N_{h} (p_{n+1}))\\& - (\N_{\M} \circ f)_{\l_1,\cdots,\hat{\l_n},\cdots,\l_{n}}(p_1, \N_{h} (p_2),\hat{\N_{h} (p_n)}, \cdots, \N_{h} (p_{n+1}))\\&
- (\N_{\M} \circ f)_{\l_1,\cdots,\hat{\l_n},\cdots,\l_{n}}(\N_{h} (p_1), p_2, \hat {\N_{h} (p_n)}, \cdots, \N_{h} (p_{n+1}))\\& - \cdots - (\N_{\M} \circ f)_{\l_1,\cdots,\hat{\l_n},\cdots,\l_{n}}(\N_{h} (p_1), \N_{h} (p_2), \cdots, \N_{h} (p_{n-1}), p_{n+1})\\&+ \N_{\M}^2(f_{\l_1,\cdots,\hat{\l_n},\cdots,\l_{n}}(p_1, p_2,p_4, \cdots, p_{n-1}, p_{n+1}))]
	\\&+ [\a^{n-1}{ (p_1)}_{\l_1}\\& \N_{\M}(f_{\l_2,\l_3,\cdots,\l_{n}}(\N_{h} (p_2), \N_{h} (p_3), \cdots, \N_{h} (p_{n+1})) \\&- (\N_{\M} \circ f)_{\l_2,\l_3,\cdots,\l_{n}}(p_2, \N_{h} (p_3), \cdots, \N_{h} (p_{n+1}))
    \\&- (\N_{\M} \circ f)_{\l_2,\l_3,\cdots,\l_{n}}(\N_{h} (p_2), p_3, \N_{h} (p_4), \cdots, \N_{h} (p_{n+1})) \\&- \cdots - (\N_{\M} \circ f)_{\l_2,\l_3,\cdots,\l_{n}}(\N_{h} (p_2), \N_{h} (p_3), \cdots, \N_{h} (p_{n}), p_{n+1})
\\&+ \N_{\M}^2(f_{\l_2,\l_3,\cdots,\l_{n}}(p_2, p_3, \cdots, p_{n+1})))]
\\& - [\a^{n-1}{p_2}_{\l_2} \\& \N_{\M} f_{\l_1,\l_3,\cdots,\l_{n}}(\N_{h} (p_1), \N_{h} (p_3), \cdots, \N_{h} (p_{n+1})) \\&- (\N_{\M} \circ f)_{\l_1,\l_3,\cdots,\l_{n}}(p_1, \N_{h} (p_3), \cdots, \N_{h} (p_{n+1}))
\\& - (\N_{\M} \circ f)_{\l_1,\l_3,\cdots,\l_{n}}(\N_{h} (p_1), p_3, \N_{h} (p_4), \cdots, \N_{h} (p_{n+1})) \\&- \cdots - (\N_{\M} \circ f)_{\l_1,\l_3,\cdots,\l_{n}}(\N_{h} (p_1), \N_{h} (p_3), \cdots, \N_{h} (p_{n}), p_{n+1})
\\&+ \N_{\M}^2(f_{\l_1,\l_3,\cdots,\l_{n}}(p_1, p_3, \cdots, p_{n+1}))] \\&
+ [\a^{n-1}{ p_3}_{\l_3 }
\\&\N_{\M} f_{\l_1,\l_2,\l_4\cdots,\l_{n}}(\N_{h} (p_1), \N_{h} (p_2),\N_{h} (p_4) \cdots, \N_{h} (p_{n+1})) \\&- (\N_{\M} \circ f)_{\l_1,\cdots,\hat{\l_3},\cdots,\l_{n}}(p_1, \N_{h} (p_2),\hat{\N_{h} (p_3)}, \cdots, \N_{h} (p_{n+1}))
\\&- (\N_{\M} \circ f)_{\l_1,\cdots,\hat{\l_3},\cdots,\l_{n}}(\N_{h} (p_1), p_2, \hat {\N_{h} (p_3)}, \cdots, \N_{h} (p_n))  \\&- \cdots- (\N_{\M} \circ f)_{\l_1,\cdots,\hat{\l_3},\cdots,\l_{n}}(\N_{h} (p_1), \N_{h} (p_2), \cdots, \N_{h} (p_{n}), p_{n+1})
\\&+ \N_{\M}^2(f_{\l_1,\cdots,\hat{\l_3},\cdots,\l_{n}}(p_1, p_2, p_4, \cdots, p_{n+1}))]\\& 
- \cdots+ (-1)^{n+1}[\a^{n-1}{p_n}_{\l_n}
\\&\N_{\M} f_{\l_1,\l_2,\cdots,\l_{n-1}}(\N_{h} (p_1), \N_{h} (p_2), \cdots, \N_{h} (p_{n-1}),\N_{h} (p_{n+1}))\\& - (\N_{\M} \circ f)_{\l_1,\cdots,\hat{\l_n},\cdots,\l_{n}}(p_1, \N_{h} (p_2),\hat{\N_{h} (p_n)}, \cdots, \N_{h} (p_{n+1}))
\\&- (\N_{\M} \circ f)_{\l_1,\cdots,\hat{\l_n},\cdots,\l_{n}}(\N_{h} (p_1), p_2, \hat {\N_{h} (p_n)}, \cdots, \N_{h} (p_{n+1})) \\&- \cdots - (\N_{\M} \circ f)_{\l_1,\cdots,\hat{\l_n},\cdots,\l_{n}}(\N_{h} (p_1), \N_{h} (p_2), \cdots, \N_{h} (p_{n-1}), p_{n+1})
\\&+ \N_{\M}^2(f_{\l_1,\cdots,\hat{\l_n},\cdots,\l_{n}}(p_1, p_2,p_4, \cdots, p_{n-1}, p_{n+1}))]
\\& -\N_{\M} [\a^{n-1}{ p_1}_{\l_1} \\&(f_{\l_2,\l_3,\cdots,\l_{n}}(\N_{h} (p_2), \N_{h} (p_3), \cdots, \N_{h} (p_{n+1})) \\&- (\N_{\M} \circ f)_{\l_2,\l_3,\cdots,\l_{n}}(p_2, \N_{h} (p_3), \cdots, \N_{h} (p_{n+1}))\\&
- (\N_{\M} \circ f)_{\l_2,\l_3,\cdots,\l_{n}}(\N_{h} (p_2), p_3, \N_{h} (p_4), \cdots, \N_{h} (p_{n+1})) \\&- \cdots - (\N_{\M} \circ f)_{\l_2,\l_3,\cdots,\l_{n}}(\N_{h} (p_2), \N_{h} (p_3), \cdots, \N_{h} (p_{n}), p_{n+1})
\\&+ \N_{\M}^2(f_{\l_2,\l_3,\cdots,\l_{n}}(p_2, p_3, \cdots, p_{n+1})))]\\&
+\N_{\M} [\a^{n-1}{p_2}_{\l_2} \\&f_{\l_1,\l_3,\cdots,\l_{n}}(\N_{h} (p_1), \N_{h} (p_3), \cdots, \N_{h} (p_{n+1})) \\&- (\N_{\M} \circ f)_{\l_1,\l_3,\cdots,\l_{n}}(p_1, \N_{h} (p_3), \cdots, \N_{h} (p_{n+1}))\\&
- (\N_{\M} \circ f)_{\l_1,\l_3,\cdots,\l_{n}}(\N_{h} (p_1), p_3, \N_{h} (p_4), \cdots, \N_{h} (p_{n+1})) \\&- \cdots - (\N_{\M} \circ f)_{\l_1,\l_3,\cdots,\l_{n}}(\N_{h} (p_1), \N_{h} (p_3), \cdots, \N_{h} (p_{n}), p_{n+1})
\\&+ \N_{\M}^2(f_{\l_1,\l_3,\cdots,\l_{n}}(p_1, p_3, \cdots, p_{n+1}))]\\&
- \N_{\M}[\a^{n-1}{ p_3}_{\l_3 }\\& f_{\l_1,\l_2,\l_4\cdots,\l_{n}}(\N_{h} (p_1), \N_{h} (p_2),\N_{h} (p_4) \cdots, \N_{h} (p_{n+1}))\\& - (\N_{\M} \circ f)_{\l_1,\cdots,\hat{\l_3},\cdots,\l_{n}}(p_1, \N_{h} (p_2),\hat{\N_{h} (p_3)}, \cdots, \N_{h} (p_{n+1}))\\&
- (\N_{\M} \circ f)_{\l_1,\cdots,\hat{\l_3},\cdots,\l_{n}}(\N_{h} (p_1), p_2, \hat {\N_{h} (p_3)}, \cdots, \N_{h} (p_n)) \\&- \cdots 
- (\N_{\M} \circ f)_{\l_1,\cdots,\hat{\l_3},\cdots,\l_{n}}(\N_{h} (p_1), \N_{h} (p_2), \cdots, \N_{h} (p_{n}), p_{n+1})
\\&+ \N_{\M}^2(f_{\l_1,\cdots,\hat{\l_3},\cdots,\l_{n}}(p_1, p_2,p_4, \cdots, p_{n+1}))]\\&
- \cdots
 + (-1)^{n+1}\N_{\M} [\a^{n-1}{p_n}_{\l_n}\\& f_{\l_1,\l_2,\cdots,\l_{n-1}}(\N_{h} (p_1), \N_{h} (p_2), \cdots, \N_{h} (p_{n-1}),\N_{h} (p_{n+1})) \\&- (\N_{\M} \circ f)_{\l_1,\cdots,\hat{\l_n},\cdots,\l_{n}}(p_1, \N_{h} (p_2),\hat{\N_{h} (p_n)}, \cdots, \N_{h} (p_{n+1}))\\&
- (\N_{\M} \circ f)_{\l_1,\cdots,\hat{\l_n},\cdots,\l_{n}}(\N_{h} (p_1), p_2, \hat {\N_{h} (p_n)}, \cdots, \N_{h} (p_{n+1})) \\&- \cdots
 - (\N_{\M} \circ f)_{\l_1,\cdots,\hat{\l_n},\cdots,\l_{n}}(\N_{h} (p_1), \N_{h} (p_2), \cdots, \N_{h} (p_{n-1}), p_{n+1})\\&+ \N_{\M}^2(f_{\l_1,\cdots,\hat{\l_n},\cdots,\l_{n}}(p_1, p_2,p_4, \cdots, p_{n-1}, p_{n+1}))] \\&+(-1)^{n+1}[\N_{\M}(f_{\l_1,\l_2,\cdots,\l_{n-1}}(\N_{h} (p_1), \N_{h} (p_2), \cdots, \N_{h} (p_n)) \\&- (\N_{\M} \circ f)_{\l_1,\l_2,\cdots,\l_{n-1}}(p_1, \N_{h} (p_2), \cdots, \N_{h} (p_n))
\\&- (\N_{\M} \circ f)_{\l_1,\l_2,\cdots,\l_{n-1}}(\N_{h} (p_1), p_2, \N_{h} (p_3), \cdots, \N_{h} (p_n)) \\&- \cdots
- (\N_{\M} \circ f)_{\l_1,\l_2,\cdots,\l_{n-1}}(\N_{h} (p_1), \N_{h} (p_2), \cdots, \N_{h} (p_{n-1}), p_n)
\\&+ \N_{\M}^2(f_{\l_1,\l_2,\cdots,\l_{n-1}}(p_1, p_2, \cdots, p_n)))_{\l_1+\l_2+\cdots+\l_n} \a^{n-1}(p_{n+1})]\\&+(-1)^{n+1}[(f_{\l_1,\l_2,\cdots,\l_{n-1}}(\N_{h} (p_1), \N_{h} (p_2), \cdots, \N_{h} (p_n)) \\&- (\N_{\M} \circ f)_{\l_1,\l_2,\cdots,\l_{n-1}}(p_1, \N_{h} (p_2), \cdots, \N_{h} (p_n))\\&
- (\N_{\M} \circ f)_{\l_1,\l_2,\cdots,\l_{n-1}}(\N_{h} (p_1), p_2, \N_{h} (p_3), \cdots, \N_{h} (p_n))\\& - \cdots - (\N_{\M} \circ f)_{\l_1,\l_2,\cdots,\l_{n-1}}(\N_{h} (p_1), \N_{h} (p_2), \cdots, \N_{h} (p_{n-1}), p_n)
\\&+ \N_{\M}^2(f_{\l_1,\l_2,\cdots,\l_{n-1}}(p_1, p_2, \cdots, p_n)))_{\l_1+\l_2+\cdots+\l_n}\N_{h}\a^{n-1}(p_{n+1})]\\&- (-1)^{n+1}\N_{\M} \circ [(f_{\l_1,\l_2,\cdots,\l_{n-1}}(\N_{h} (p_1), \N_{h} (p_2), \cdots, \N_{h} (p_n)) \\&- (\N_{\M} \circ f)_{\l_1,\l_2,\cdots,\l_{n-1}}(p_1, \N_{h} (p_2), \cdots, \N_{h} (p_n))\\&
- (\N_{\M} \circ f)_{\l_1,\l_2,\cdots,\l_{n-1}}(\N_{h} (p_1), p_2, \N_{h} (p_3), \cdots, \N_{h} (p_n)) \\&- \cdots
 - (\N_{\M} \circ f)_{\l_1,\l_2,\cdots,\l_{n-1}}(\N_{h} (p_1), \N_{h} (p_2), \cdots, \N_{h} (p_{n-1}), p_n)
\\&+ \N_{\M}^2(f_{\l_1,\l_2,\cdots,\l_{n-1}}(p_1, p_2, \cdots, p_n)))_{\l_1+\l_2+\cdots+\l_n}\a^{n-1}(p_{n+1})]
\\& - \sum_{j=2}^{n+1}(f_{\l_1+\l_j,\hat{\l_1},\l_2,\cdots,\hat{\l_j}, \cdots,\l_{n}}\\&(\N_{h}([{\N_{h}(p_1)}_{\l_1} p_j]+[{p_1}_{\l_1}\N_{h}(p_j)]-\N_{h}[{p_1}_{\l_1} {p_j}]), \hat{\N_{h} \a(p_1)}, \N_{h} \a(p_2), \cdots,\hat{\N_{h} \a(p_j)},\cdots, \N_{h} \a(p_{n+1}))
\\&+(\N_{\M} \circ f)_{\l_1+\l_j,\hat{\l_1},\l_2,\cdots,\hat{\l_j}, \cdots,\l_{n}}
\\&(([{\N_{h}(p_1)}_{\l_1} p_j]+[{p_1}_{\l_1}\N_{h}(p_j)]-\N_{h}[{p_1}_{\l_1}{p_j}]), \hat{\N_{h} \a(p_1)}, \N_{h} \a(p_2), \cdots,\hat{\N_{h} \a(p_j)},\cdots, \N_{h} \a(p_{n+1}))
\\&+(\N_{\M} \circ f)_{\l_1+\l_j,\hat{\l_1},\l_2,\cdots,\hat{\l_j}, \cdots,\l_{n}}\\&(\N_{h}([{\N_{h}(p_1)}_{\l_1} p_j]+[{p_1}_{\l_1}\N_{h}(p_j)]-\N_{h}[{ p_1}_{\l_1}{p_j}]), \hat{\N_{h} \a(p_1)}, \a(p_2), \cdots,\hat{\N_{h} \a(p_j)},\cdots, \N_{h} \a(p_{n+1}))\\& - \cdots 
- (\N_{\M} \circ f)_{\l_1+\l_j,\hat{\l_1},\l_2,\cdots,\hat{\l_j}, \cdots,\l_{n}}
\\&(\N_{h}([{\N_{h}(p_1)}_{\l_1}p_j]+[{p_1}_{\l_1}\N_{h}(p_j)]-\N_{h}[{p_1}_{\l_1}{ p_j}]), \hat{\N_{h} \a(p_1)}, \N_{h} \a(p_2), \cdots,\hat{\N_{h}\a(p_j)},\cdots, \a(p_{n+1}))
\\&- \N_{\M}^2(f_{\l_1+\l_j,\hat{\l_1},\l_2,\cdots,\hat{\l_j}, \cdots,\l_{n}}
\\&([{\N_{h}(p_1)}_{\l_1} p_j]+[{p_1}_{\l_1}\N_{h}(p_j)]-\N_{h}[{ p_1}_{\l_1}{ p_j}],\hat{\a(p_1)}, \a(p_2), \cdots,\hat{\a(p_j)},\cdots, \a(p_{n+1}))))
\\&+ \sum_{j=3}^{n+1}(f_{\l_2+\l_j,\l_1,\hat{\l_2},\cdots,\hat{\l_j}, \cdots,\l_{n}}
\\&(\N_{h}([{\N_{h}(p_2)}_{\l_2}p_j]+[{p_2}_{\l_2}\N_{h}(p_j)]-\N_{h}[{p_2}_{\l_2}{p_j}]), \N_{h} \a(p_1), \hat{\N_{h} \a(p_2)}, \N_{h} \a(p_3), \cdots,\hat{\N_{h} \a(p_j)},\cdots, \N_{h} \a(p_{n+1}))
\\&-(\N_{\M} \circ f)_{\l_2+\l_j,\l_1,\hat{\l_2},\cdots,\hat{\l_j}, \cdots,\l_{n}}
\\&(([{\N_{h}(p_2)}_{\l_2}p_j]+[ {p_2}_{\l_2}\N_{h}(p_j)]-\N_{h}[{p_2}_{\l_2}{p_j}]), \N_{h} \a(p_1),\hat{\N_{h}\a(p_2)}, \cdots,\hat{\N_{h} \a(p_j)},\cdots, \N_{h} \a(p_{n+1}))
\\&-(\N_{\M} \circ f)_{\l_2+\l_j,\l_1,\hat{\l_2},\cdots,\hat{\l_j}, \cdots,\l_{n}}
\\&(\N_{h}([{\N_{h}(\a(p_2))}_{\l_2}p_j]+[{p_2}_{\l_2}\N_{h}(p_j)]-\N_{h}[{ p_2}_{\l_2}{p_j}]),\N_{h}\a(p_1), \hat{\N_{h}\a(p_2)}, \a(p_3), \cdots,\hat{\N_{h} \a(p_j)},\cdots, \N_{h} \a(p_{n+1})) \\&- \cdots
 - (\N_{\M} \circ f)_{\l_2+\l_j,\l_1,\hat{\l_2},\cdots,\hat{\l_j}, \cdots,\l_{n}}
\\&(\N_{h}([{\N_{h}(p_2)}_{\l_2}\p_j]+[{p_2}_{\l_2}\N_{h}(p_j)]-\N_{h}[{p_2}_{\l_2}{ p_j}]), \N_{h} \a(p_1), \hat{\N_{h} \a(p_2)}, \N_{h} \a(p_3), \cdots,\hat{\N_{h} \a(p_j)},\cdots, \a(p_{n+1})) \\&+ \N_{\M}^2(f_{\l_2+\l_j,\l_1,\hat{\l_2},\cdots,\hat{\l_j}, \cdots,\l_{n}}
\\&([{\N_{h}(p_2)}_{\l_2}p_j]+[{ p_2}_{\l_2}\N_{h}(p_j)]-\N_{h}[{p_2}_{\l_2}{p_j}],\a(p_1),\hat{\a(p_2)}, \a(p_3), \cdots,\hat{ \a(p_j)},\cdots, \a(p_{n+1}))))\\&- \sum_{j=4}^{n+1}(f_{\l_3+\l_j,\l_1,\l_2, \hat{\l_3},\cdots,\hat{\l_j}, \cdots,\l_{n}}\\&(\N_{h}([{\N_{h}(p_3)}_{\l_3} p_j]+[{p_3}_{\l_3}\N_{h}(p_j)]-\N_{h}[{p_3}_{\l_3} {p_j}]), \N_{h} \a(p_1), \hat{\N_{h} \a(p_3)}, \N_{h}\a(p_2), \cdots,\hat{\N_{h} \a(p_j)},\cdots, \N_{h} \a(p_{n+1}))\\&+(\N_{\M} \circ f)_{\l_3+\l_j,\l_1,\l_2, \hat{\l_3},\cdots,\hat{\l_j}, \cdots,\l_{n}}\\&(([{\N_{h}(p_3)}_{\l_3} p_j]+[{p_3}_{\l_3}\N_{h}(p_j)]-\N_{h}[{p_3}_{\l_3}{p_j}]), \N_{h} \a(p_1),\N_{h}\a(p_2), \hat{\N_{h} \a(p_3)}, \cdots,\hat{\N_{h} \a(p_j)},\cdots, \N_{h} \a(p_{n+1}))\\&
+(\N_{\M} \circ f)_{\l_3+\l_j,\l_1,\l_2, \hat{\l_3},\cdots,\hat{\l_j}, \cdots,\l_{n}}\\&(\N_{h}([{\N_{h}(p_3)}_{\l_3} p_j]+[ {p_3}_{\l_3}\N_{h}(p_j)]-\N_{h}[{p_3}_{\l_3}{p_j}]),\N_{h}\a(p_1), \N_{h}\a(p_2), \hat{\N_{h} \a(p_3)}, \a(p_4), \cdots,\hat{\N_{h} \a(p_j)},\cdots, \N_{h} \a(p_{n+1})) \\&- \cdots - (\N_{\M} \circ f)_{\l_3+\l_j,\l_1,\l_2, \hat{\l_3},\cdots,\hat{\l_j}, \cdots,\l_{n}}\\&(\N_{h}([{\N_{h}(p_3)}_{\l_3} p_j]+[{p_3}_{\l_3}\N_{h}(p_j)]-\N_{h}[{p_3}_{\l_3}{p_j}]), \N_{h} \a(p_1), \N_{h} \a(p_2), \hat{\N_{h} \a(p_3)}, \N_{h} \a(p_4), \cdots,\hat{\N_{h} \a(p_j)},\cdots, \a(p_{n+1}))\\&- \N_{\M}^2(f_{\l_3+\l_j,\l_1,\l_2, \hat{\l_3},\cdots,\hat{\l_j}, \cdots,\l_{n}}\\&([{\N_{h}(p_3)}_{\l_3}p_j]+[{p_3}_{\l_3}\N_{h}(p_j)]-\N_{h}[{p_3}_{\l_3}{ p_j}],\a(p_1),\a(p_2),\hat{\a(p_3)}, \a(p_4), \cdots,\hat{\a(p_j)},\cdots, \a(p_{n+1}))))
\\&\cdots
\\&+ (-1)^n\sum_{j=n+1}^{n+1}(f_{\l_n+\l_j,\l_1,\l_2,\cdots,\hat{\l_j}, \cdots,\l_{n-1}}(\N_{h}([{\N_{h}(p_n)}_{\l_n}p_j]+[{p_n}_{\l_n}\N_{h}(p_j)]-\N_{h}[ {p_n}_{\l_n} {p_j}]), \N_{h} \a(p_1),\\&\N_{h}\a(p_2),  \N_{h} \a(p_3), \cdots,\hat{\N_{h}\a(p_j)},\cdots, \hat{\N_{h} \a(p_n)}, \N_{h} \a(p_{n+1}))
\\&-(\N_{\M} \circ f)_{\l_n+\l_j,\l_1,\l_2,\cdots,\hat{\l_j}, \cdots,\l_{n-1}}(([{\N_{h}(p_n)}_{\l_n} p_j]+[ {p_n}_{\l_n}\N_{h}(p_j)]-\N_{h}[{p_n}_{\l_n}{p_j}]), \N_{h} \a(p_1), \cdots,
\\&\hat{\N_{h} \a(p_j)},\cdots,\hat{\N_{h} \a(p_n)}, \N_{h} \a(p_{n+1}))
\\&-(\N_{\M} \circ f)_{\l_n+\l_j,\l_1,\l_2,\cdots,\hat{\l_j}, \cdots,\l_{n-1}}(\N_{h}([{\N_{h}(p_n)}_{\l_n} p_j]+[{p_n}_{\l_n}\N_{h}(p_j)]-\N_{h}[{p_n}_{\l_n}{ p_j}]), \N_{h} \a(p_1), \N_{h} \a(p_2),
\\& \a(p_3), \cdots,\hat{\N_{h} \a(p_j)},\cdots,\hat{\N_{h} \a(p_n)}, \N_{h} \a(p_{n+1})) \\&- \cdots
 - (\N_{\M} \circ f)_{\l_n+\l_j,\l_1,\l_2,\cdots,\hat{\l_j}, \cdots,\l_{n-1}}(\N_{h}([{\N_{h}(p_n)}_{\l_n}p_j]+ [{p_n}_{\l_n}\N_{h}(p_j)]-\N_{h}[{p_n}_{\l_n}{p_j}]), \N_{h} \a(p_1), \N_{h} \a(p_2),
\\&\N_{h} \a(p_3), \cdots,\hat{\N \a(p_j)},\cdots,\hat{\N_{h} \a(p_n)}, \a(p_{n+1}))
\\&+ \N_{\M}^2(f_{\l_n+\l_j,\l_1,\l_2,\cdots,\hat{\l_j}, \cdots,\l_{n-1}}([{\N_{h}(p_n)}_{\l_n} p_j]+[{p_n}_{\l_n}\N_{h}(p_j)]-\N_{h}[{p_n}_{\l_n}{p_j}],\a(p_1),\a(p_2),
\\& \a(p_3), \cdots,\hat{\a(p_j)},\cdots,\hat{\a(p_n)}, \a(p_{n+1})))).
\end{align*} By using the Definition \eqref{defrepNIJ}, we obtain the desired result that verifies \begin{align*}
    \phi^{n+1}(\delta_{HomL}^n(f))(p_1, p_2, p_3, \cdots , p_{n+1}) = \partial^n_{HN}(\phi^n(f))(p_1, p_2, p_3, \cdots , p_{n+1}).
\end{align*}
\end{proof}
The above Lemma can be expressed by the following commutative diagram:\\
\begin{tikzcd}
C^1_{HomL}(\L,\M) \arrow[r,"\delta_{HomL}^1"] \arrow[d, "\phi^1"'] & C^2_{HomL}(\L,\M) \arrow[r,"\delta_{HomL}^2"] \arrow[d, "\phi^2"'] & \dots \arrow[r,"\delta_{HomL}^{n-1}"] & C^n_{HomL}(\L,\M) \arrow[r,"\delta_{HomL}^{n}"] \arrow[d, "\phi^n"'] & C^{n+1}_{HomL}(\L,\M) \arrow[d, "\phi^{n+1}"'] \dots \\
C^1_{HN}(\L,\M) \arrow[r,"\partial_{HN}^1"] & C^2_{HN}(\L,\M) \arrow[r,"\partial_{HN}^2"] & \dots \arrow[r,"\partial_{HN}^{n-1}"] & C^n_{HN}(\L,\M) \arrow[r,"\partial_{HN}^n"] & C^{n+1}_{HN}(\L,\M)\dots .
\end{tikzcd}
\par Next, we unify the cochain complex of Hom-Leibniz conformal algebra and the cochain complex of Hom-Nijenhuis operator and form a new cochain complex of Hom-Nijenhuis Leibniz conformal algebra. 
\begin{defn}\label{def3.3}
	Let $\L_{\N_h}$ be a Hom-Nijenhuis Leibniz conformal algebra with a representation $(\M, l, r,\b, \N_{\M})$. Let's define
	\begin{align}
	    C^0_{HNLA}(\L,\M) &:= C^0_{HomL}(\L,\M)\\
	     C^n_{HNLA}(\L,\M) &:= C^n_{HomL}(\L,\M) \oplus C^{n-1}_{HN}(\L,\M), \quad \forall\quad  n \geq 1,
	 \end{align}
	and a coboundary map $d_{HNLA}^n: C^n_{HNLA}(\L,\M) \rightarrow C^{n+1}_{HNLA}(\L,\M)$, for any $f \in C^n_{HomL}(\L,\M)$, $g \in C^{n-1}_{HN}(\L,\M)$ by
	\begin{align}
	    d_{HNLA}^n(f, g) = (\delta_{HomL}^n(f), -\partial_{HN}^{n-1}(g) - \phi^n(f)).
	\end{align}
\end{defn}
\begin{thm}\label{thm3.4}
	The map $d_{HNLA}^n: C^n_{HNLA}(\L,\M) \rightarrow C^{n+1}_{HNLA}(\L,\M)$, satisfies $d_{HNLA}^n \circ d_{HNLA}^{n+1} = 0$.
\end{thm}
\begin{proof} For any $f \in C^n_{HomL}(\L,\M)$ and $g \in C^{n-1}_{HN}(\L,\M)$,
\begin{align*}
 d_{HNLA}^{n+1} \circ d_{HNLA}^n(f, g) &= d_{HNLA}^{n+1} \left(\delta_{HomL}^n(f), -\partial^{n-1}_{HN}(g) - \phi^n(f) \right) \\&
 = \left(\delta_{HomL}^{n+1}(\delta_{HomL}^n(f)), -\partial_{HN}^{n} \left(-\partial_{HN}^{n-1}(g) - \phi^n(f) \right) - \phi^{n+1}(\delta_{HomL}^n(f)) \right)\\&
 = \left(\delta_{HomL}^{n+1}(\delta_{HomL}^n(f)), \partial_{HN}^{n}\partial_{HN}^{n-1}(g) + \partial_{HN}^{n}\phi^n(f) - \phi^{n+1}(\delta_{HomL}^n(f)) \right)\\&
= \left(0, 0 \right).
\end{align*} \end{proof}From Definition \eqref{def3.3} and Theorem \eqref{thm3.4}, we see that $\{C^n_{HNLA}(\L,\M), d_{HNLA}^n\}$ forms a cochain complex of the Hom-Nijenhuis Leibniz conformal algebra $\L_{\N_h}$ with representation $(\M, l, r, \b, \N_{\M})$. Its corresponding cohomology is denoted by $H^n_{HNLA}(\L,\M)$ for $n>0$. The above cochain complex can be expressed in terms of a short exact sequence as follows:
\begin{center}
	\begin{tikzcd}
	0\arrow[r]& C^n_{HN}(\L,\M) \arrow[r]& C^n_{HNLA}(\L,\M) \arrow[r] & C^n_{HomL}(\L,\M) \arrow[r] & 0.
	\end{tikzcd} 
\end{center} 
\section{ Deformation of Hom-Nijenhuis-Leibniz conformal algebra}
In this section, we study the formal one-parameter deformation of Hom-Nijenhuis-Leibniz conformal algebra. To avoid difficulty in computation in this section, we denote the bracket $[\cdot_\l \cdot]_\L$ by $\{\cdot_\l \cdot\}$ and $\N_h$ by $\N$. 
\begin{defn} A one-parameter formal deformation of a Hom-Nijenhuis Leibniz conformal algebra $(\L_{\N_h}, \{\cdot_\l \cdot\},\a)$ is a pair of two power series $(\{\cdot_\l \cdot\}_t, \N_t)$
 
	 \begin{align*}
	     \{\cdot_\l \cdot\}_t &= \sum_{i=0}^{\infty} \{\cdot_\l \cdot\}_i t^i, \quad \{\cdot_\l \cdot\}_i \in C^2_{HomL}(\L,\L),\\
	 \N_t & \sum_{i=0}^{\infty} \N_i t^i, \quad \N_i \in C^1_{HN}(\L,\L),
	 \end{align*}
	 
	such that $(\L[[t]]_{\N_t}, \{\cdot_\l \cdot\}_t,\a)$ is a Hom-Nijenhuis-Leibniz conformal algebra with $(\{\cdot_\l \cdot\}_0, \N_0) = (\{\cdot_\l \cdot\}, \N)$, where $\L[[t]]$ which is the space of formal power series in $t$ with coefficients from $\L$, is a $\mathbb{C}[[t]]$ module, $\mathbb{C}$ being the ground field of $(\L_{\N}, \{\cdot_\l \cdot\},\a)$.
\end{defn}The above definition holds if and only if for any $p, q, r \in \L$, the following conditions are satisfied:
\begin{align*}\a\circ \N_t&=\N_t\circ \a\\
 \{\a(p)_\l \{q_\m r\}_t\}_t&= {\{\{p_\l q\}_t}_{\l+\m} \a(r) \}_t+ \{\a(q)_\m \{p_\l r\}_t\}_t, \\
\{\N_t(p)_\l \N_t(q)\}_t &= \N_t \left(\{p_\l \N_t(q)\}_t + \{\N_t(p)_\l q\}_t- \N_t(\{p_\l q\}_t) \right).
\end{align*}
Expanding the above equations and equating the coefficients of $t^n$ from both sides we have
\begin{align}\label{eqhom}
\sum_{i+j=n} \{\a(p)_\l \{q_\m r\}_j\}_i&= \sum_{i+j=n} {\{\{p_\l q\}_j}_{\l+\m} \a(r) \}_i+ \{\a(q)_\m \{p_\l r\}_j\}_i, \\
\label{eqnij}\sum_{i+j+k=n}\{\N_j(p)_\l \N_k(q)\}_i &= \sum_{i+j+k=n}\N_i \left(\{p_\l \N_j(q)\}_k + \{\N_j(p)_\l q\}_k- \N_j(\{p_\l q\}_k) \right).
\end{align} Note that when $n = 0$, the above conditions are exactly the conditions in the definitions of Hom-Leibniz conformal algebra and the Hom-Nijenhuis operator see Definition \eqref{defdef} and \eqref{Nijenhuis}.
\begin{prop} Let $(\{\cdot_\l \cdot\}_t, \N_t, \a)$ be a one-parameter deformation of a Hom-Nijenhuis-Leibniz conformal algebra $(\L_{\N}, \{\cdot_\l \cdot\},\a)$. Then $(\{\cdot_\l \cdot\}_1, \N_1)$ is a $2$-cocycle in the cochain complex $\{C^n_{HNLA}(\L,\L), d_{HNLA}^n\}$ of Hom-Nijenhuis Leibniz conformal algebra.\end{prop}
\begin{proof}
Note that the pair $(\{\cdot_\l \cdot\}_1, \N_1)$ is $2$ cocycle if $d_{HNLA}^2(\{\cdot_\l \cdot\}_1, \N_1)=0$. By expanding the L.H.S of this expression we obtain \begin{align}\label{eq2cocycle}
	d_{HNLA}^2(\{\cdot_\l \cdot\}_1, \N_1)= (\delta_{HomL}^2(\{\cdot_\l \cdot\}_1), -\partial_{HN}^{1}(\N_1)-\phi^2(\{\cdot_\l \cdot\}_1)). \end{align} First consider Eq. \eqref{eqhom} for $n=1$, we have
\begin{align*}
\{\a(p)_\l \{q_\m r\}_1\} + \{\a(p)_\l \{q_\m r\}\}_1=& \{{\{p_\l q\}_1}_{\l+\m} \a(r)\} + \{\{p_\l q\}_{\l+\m} \a(r)\}_1 + \{\a(q)_\m \{p_\l r\}\}_1 + \{\a(q)_\m \{p_\l r\}_1\}.
\end{align*}
By Eq. \eqref{eqdeltahom2} above expression is equivalent to $)= \delta_{HomL}^2(\{\cdot_\l \cdot\}_1)(p, q, r)\in C^2_{HomL}(\L,\L)$.
\newline Now, consider Eq \eqref{eqnij} for $n = 1$, we have
\begin{align*}&\{\N p_\l \N q\}_1 + \{\N_1p_\l \N q\} + \{\N p_\l \N_1q\} \\
&= \N_1(\{\N p_\l q\}) + \N(\{\N p_\l q\}_1) + \N(\{{\N_1p}_\l q\}) \\
	&\quad + \N_1(\{p_\l \N q\}) + \N(\{p_\l \N q\}_1) + \N(\{p_\l \N_1q\}) \\
	&\quad - \N_1 \N(\{p_\l q\}) - \N^2(\{p_\l q\}_1)- \N \N_1(\{p_\l q\}).
	\end{align*}Re-arranging above equation
 	\begin{align*} &\N \N_1(\{p_\l q\})+ \{\N p_\l \N q\}_1- \N(\{\N p_\l q\}_1) - \N(\{p_\l \N q\}_1) + \N^2(\{p_\l q\}_1)\\=& - \N \N_1(\{p_\l q\}) -\{\N_1p_\l \N q\} - \{\N p_\l \N_1q\}+ \N_1(\{\N p_\l q\})\\&+ \N(\{{\N_1p}_\l q\})+\N_1(\{p_\l \N q\})+\N(\{p_\l \N_1 q\}),
	\end{align*}or
		\begin{align*} &\N \N_1(\{p_\l q\})+ \{\N p_\l \N q\}_1- \N(\{\N p_\l q\}_1) - \N(\{p_\l q\}_1) + \N^2(\{p_\l q\}_1)-\{p_\l \N\N_1 q\}-\{\N\N_1 p_\l q\}\\=& - \N \N_1(\{p_\l q\}) -\{\N_1p_\l \N q\} - \{\N p_\l \N_1q\}+ \N_1(\{\N p_\l q\})+ \N(\{{\N_1p}_\l q\})\\&+\N_1(\{p_\l \N q\})+\N(\{p_\l \N_1 q\})-\{p_\l \N\N_1 q\}-\{\N\N_1 p_\l q\}.
	\end{align*}
Which implies that $\phi^2(\{\cdot_\l \cdot\}_1)(p,q)=-\partial_{HN}^{1}(\N_1)(p,q)$. 
	Hence, by Eq. \eqref{eq2cocycle} we obtain $d_{HNLA}^2(\mu_1, N_1) = 0$. This completes the proof.
	\end{proof}
\begin{defn}
	Let $(\{\cdot_\l \cdot\}_t, \N_t,\a)$ and $(\{\cdot_\l \cdot\}'_t, \N'_t,\a)$ be two deformations of a Hom-Nijenhuis Leibniz conformal algebra $(\L_{\N}, \{\cdot_\l \cdot\},\a)$. A formal isomorphism from $(\{\cdot_\l \cdot\}_t, \N_t,\a)$ to $(\{\cdot_\l \cdot\}'_t, \N'_t,\a)$ is a power series $\psi_t = \sum_{i=0}^\infty \psi_i t^i : \L[[t]] \to \L[[t]]$, where $\psi_i : \L\to \L$ are linear maps with $\psi_0 = \text{Id}_\L$ satisfying the following conditions:
 \begin{align*}
\a\N_t&=\N_t \a,\\ \psi_t \circ \{\cdot_\l \cdot\}'_t &= \{\cdot_\l \cdot\}_t \circ (\psi_t \otimes \psi_t), \\
	\psi_t \circ \N'_t &= \N_t \circ \psi_t.
	\end{align*}
\end{defn}
If there exists a formal isomorphism $\psi_t$ between two formal deformations $(\{\cdot_\l \cdot\}_t, \N_t,\a)$ and $(\{\cdot_\l \cdot\}'_t, \N'_t,\a)$, we say these formal deformations are equivalent. From the above definition, we also have
\begin{align*}
\sum_{i+j=n;i,j\geq 0} \psi_i(\{p_\l q\}'_j) = \sum_{i+j+k=n;i,j,k \geq 0} \{\psi_j(p)_\l \psi_k(q)\}_i, \quad p,q \in \L
\end{align*} and \begin{align*}
\sum_{i+j=n;i,j \geq 0} \psi_i \circ \N'_j = \sum_{i+j=n; i,j \geq 0} \N_i \circ \psi_j.
\end{align*}
For $ n=1 $ and $\psi_0=Id_\L$, we have
\begin{align*}
\psi_1(\{p_\l q\}') +\{p_\l q\}'_1 = \{\psi_1(p)_\l q\}+\{p_\l \psi_1(q)\}+\{p_\l q\}_1, \quad p,q \in \L
\end{align*} and \begin{align*}
 \psi_1 \circ \N' + \N'_1= \N \circ \psi_1+ \N_1 .
\end{align*} Rearranging above equations, we get \begin{align*}
\{p_\l q\}'_1-\{p_\l q\}_1 = - \psi_1(\{p_\l q\}')+ \{\psi_1(p)_\l q\}+\{p_\l \psi_1(q)\}, \quad p,q \in \L
\end{align*} and \begin{align*} \N'_1- \N_1= \N \circ \psi_1- \psi_1 \circ \N'.
\end{align*}This implies that $(\{p_\l q\}'_1, \N'_1)-(\{p_\l q\}_1,\N_1) = d^1_{HNLA}(\psi_1,0)\in C^1_{HNLA}(\L,\L)$. Thus, the infinitesimal of two equivalent formal deformations of Hom-Nijenhuis Leibniz conformal algebras are cohomologous. For more detail about infinitesimal see Definition (5.8) of \cite{AWMB}.
\begin{thm} A Hom-Nijenhuis Leibniz conformal algebra $(\L_{\N},\{\cdot_\l \cdot\},\a) $ is called rigid , if $H^n_{HNLA}(\L,\L)=0,$ for $ n=2$.
\end{thm}
\begin{proof}Proof of the theorem uses the same argument as in the paper \cite{BR}(Theorem 5.5).\end{proof}

  \section{Hom-$NS$-Leibniz conformal algebras and Twisted Rota-Baxter operators}
  \begin{defn}
	A $\mathbb{C}[\p]$-module $\L$ is called Hom-$NS$-Leibniz conformal algebra together with the conformal sesqui-linear maps $\triangleleft_{\l},  \triangleright_{\l}, \vee_\l: \L \otimes \L \to \L[\l]$ in which $\L$ is skew-symmetric and a linear map $\a : \L \to \L$ that satisfies $$\a(p\lt_{\l}q)=\a(p)\lt_\l \a(q),\quad \a(p\rt_{\l}q)=\a(p)\rt_\l \a(q)\quad \a(p\vee_{\l}q)= \a(p)\vee_\l \a(q)$$ and the following identities
\begin{equation}\label{eq28}
\begin{aligned}
\a(p)\rt_\l(q*_\m r)&= (p\rt_\l q)\rt_{\l+\m} \a(r)+\a(q)\lt_\m (a\rt_\l c),\\
\a(p)\lt_\l(q\rt_\m r)&= (p\lt_\l q)\rt_{\l+\m} \a(r)+\a(q)\rt_\m (a*_\l c),\\
\a(p)\lt_\l(q\lt_\m r)&= (p*_\l q)\lt_{\l+\m} \a(r)+\a(q)\lt_\m (a\lt_\l c),
\\
\a(p)\vee_\l(q*_\m r)
-\a(q)\vee_\m(a*_\l r)
-(p*_\l q)\vee_{\l+\m} \a(r)&\\
+\a(p)\lt_\l(q\vee_\m r)
-\a(q)\lt_\m(p\vee_\l r)
-(p\vee_\l q) \rt_{\l+\m}\a(r)
&=0.\end{aligned}\end{equation}for $p,q,r \in \L$. Here $p*_\l q= p \lt_\l q +p \rt_{\l} q+ p \vee_\l q.$
\end{defn} 
We represent Hom-$NS$-Leibniz conformal algebra by $(\L, \lt_\l,\rt_\l, \vee_\l, \a)$. Note that, we can obtain $NS$-Leibniz  conformal algebra and a Hom-Leibniz conformal algebra by considering $\a=Id$ and $\lt_\l,\rt_\l=0$ respectively. 
\begin{prop}\label{prophomNSleib}
  Let $(\L, \rt_\l,\lt_\l, \vee_\l, \a)$ be a Hom-NS-Leibniz conformal algebra, Then $(\L, *_\l,\a)$ forms a Hom-Leibniz conformal algebra, Where $p*_\l q= [p_\l q]$, for all $p,q\in \L$.
\end{prop}
\begin{proof}First we observe that, \begin{align*}
    \a([p_\l q])=& \a(p\lt_\l q+p\rt_\l q+ p\vee_\l q)\\&=\a(p \lt_\l q)+ \a(q\rt_\l p)+ \a(p\vee_\l q)\\&= \a(p)\lt_\l \a(q)- \a(q)\rt_{\l} \a(p)+ \a(p)\vee_\l \a(q)\\&= [\a(p)_\l \a(q)].
\end{align*}
 Next, for any $p,q, r \in \L$, we have\begin{align*}[\a(p)_\l[q_\m r]]=&[\a(p)_\l(q\lt_\m r+ q \rt_{\m}r+q\vee_\m r)]\\=&
 [\a(p)_\l(q\lt_\m r)]
 +[\a(p)_\l (q\rt_{\m}r)]+
 [\a(p)_\l (q\vee_\m r)]\\=&
 \a(p)\lt_\l(q\lt_\m r)
 + \a(p)\rt_\l (q\lt_\m r)
 +\a(p)\vee_\l(q\lt_\m r)\\&
+ \a(p)\lt_\l(q\rt_\m r)
 + \a(p)\rt_\l (q\rt_\m r)
 +\a(p)\vee_\l(q\rt_\m r)\\&
+ \a(p)\lt_\l(q\vee_\m r)
 + \a(p)\rt_\l (q\vee_\m r)
 +\a(p)\vee_\l(q\vee_\m r).
\end{align*}Likewise, we have \begin{align*}[\a(q)_\l[p_\m r]]=&[\a(q)_\m(p\lt_\l r+ p \rt_{\l}r+p\vee_\l r)]\\=&
 [\a(q)_\m(p\lt_\l r)]
 +[\a(q)_\m (p\rt_{\l}r)]+
 [\a(q)_\m (p\vee_\l r)]\\=&
 \a(q)\lt_\m(p\lt_\l r)
 + \a(q)\rt_\m (p\lt_\l r)
 +\a(q)\vee_\m(p\lt_\l r)\\&
+ \a(q)\lt_\m(p\rt_\l r)
 + \a(q)\rt_\m (p\rt_\l r)
 +\a(q)\vee_\m(p\rt_\l r)\\&
+ \a(q)\lt_\m(p\vee_\l r)
 + \a(q)\rt_\m (p\vee_\l r)
 +\a(q)\vee_\m(p\vee_\l r),\end{align*}
 and
\begin{align*}[{[p_\l q]}_{\l+\m}\a(r)]=&
 [(p\lt_\l q+p\rt_{\l} q+ p\vee_\l q)_{\l+\m}\a(r)]\\=&
(p\lt_\l q)\lt_{\l+\m}\a(r)+(p\lt_{\l} q)\lt_{\l+\m}\a(r)+ (p\vee_\l q)\lt_{\l+\m}\a(r)\\&
 +(p\lt_\l q)\rt_{\l+\m}\a(r)+(p\lt_{\l} q)\rt_{\l+\m}\a(r)+ (p\vee_\l q)\rt_{\l+\m}\a(r)\\&
+(p\lt_\l q)\vee_{\l+\m}\a(r)+(p\lt_{\l} q)\vee_{\l+\m}\a(r)+ (p\vee_\l q)\vee_{\l+\m}\a(r).\end{align*} By using above equations and Eq. \eqref{eq28}, we see that  the following identity holds
$$[\a(p)_\l[q_\m r]]- [\a(q)_\m[p_\l r]]- [[p_\l q]_{\l+\m}\a(r)]=0.$$This completes the proof. \end{proof}
The Hom-Leibniz conformal algebra developed from Hom-$NS$-Lie conformal algebra $(\L, \lt_\l, \rt_\l, \vee_\l, \a)$ using the Proposition \eqref{prophomNSleib} is also called its adjacent Hom-Leibniz conformal algebra.\\
 Let $(\L, \rt_\l,\lt_\l,\vee_\l)$ be a $NS$-Leibniz conformal algebra. A linear map $\a: \L \to\L$ is said to be an $NS$-Leibniz conformal algebra morphism if for all $p,q\in \L$, $\a$ satisfies 
 \begin{align*}
     \a(p\lt_\l q)&=\a(p)\lt_\l \a(q),\\
     \a(p\rt_\l q)&= \a(p)\rt_\l \a(q),\\
     \a(p\vee_\l q)&= \a(p)\vee_\l \a(q).
 \end{align*} Note that, we can construct a Hom-$NS$-Leibniz conformal algebra from an $NS$-Leibniz conformal algebra by using morphism map $\a$.
\begin{prop}Let $\a: \L\to \L$ be an $NS$-Leibniz conformal algebra morphism of the $NS$-Leibniz conformal algebra $(\L, \rt_\l,\lt_\l,\vee_\l)$. Then $(\L, \rt_\l^{\a},\lt_\l^{\a},\vee_{\l}^{\a}, \a)$  is a Hom-$NS$-Leibniz conformal algebra, where $p\lt_{\l}^{\a}q= \a(p\lt_{\l}q)$ , $p\rt_{\l}^{\a}q= \a(p\rt_{\l}q)$, and $p\vee_{\l}^{\a}q=\a(p\vee_{\l}q)$.\end{prop}\begin{proof}The map $\a$ satisfies \begin{equation*}\a(p\lt_{\l}^{\a}q)=\a(p)\lt_{\l}^{\a}\a(q),\quad \a(p\rt_{\l}^{\a}q)=\a(p)\rt_{\l}^{\a}\a(q)\quad and\quad \a(p\vee_{\l}^{\a}q)=\a(p)\vee_{\l}^{\a}\a(q).\end{equation*}
Next, we notice that 
\begin{align*}
    \a(p)\rt_\l^\a(q*_\m^\a r)&= \a(\a(p)\rt_\l \a(q*_\m  r))\\&=
    \a^2(p\rt_\l (q*_\m  r))\\&= \a^2((p\rt_\l q)\rt_{\l+\m} r+ 
 q\lt_\m (a\rt_\l c))\\&= \a^2((p\rt_\l q)\rt_{\l+\m} r)+ 
\a^2( q\lt_\m (a\rt_\l c))\\&= \a(\a(p\rt_\l q)\rt_{\l+\m} \a (r))+  \a( \a (q)\lt_\m \a(a\rt_\l c))
\\&= \a( (p\rt_\l^\a q)\rt_{\l+\m}  r)+  \a(q\lt_\m (p\rt_\l^\a r))\\&= (p\rt_\l^\a q)\rt_{\l+\m}^\a  r+   q\lt_\m^\a (p\rt_\l^\a r).
\end{align*}
Similarly, we can show other identities of Hom-NS-Leibniz conformal algebra. This completes the proof.
\end{proof}
In the next propositions, we show how the Nijenhuis operator, Rota-Baxter operator, and twisted Rota-Baxter operator yield Hom-NS-Leibniz conformal algebras.
\begin{prop}\label{prop5.4}Let $\N : \L \to \L$ be a Nijenhuis operator on the Hom-Leibniz conformal algebra $(\L, [\cdot_\l \cdot], \a)$, then $(\L, \lt_\l,\rt_\l,\vee_\l, \a)$ is a Hom-$NS$-Leibniz conformal algebra, where
	$$p\lt_\l q= [\N(p)_\l q],~~~p\rt_\l q= [p_\l \N(q)]~~~ and ~~~~~p\vee_\l q= -\N[p_\l q].$$
\end{prop}
\begin{proof}As we can show that
	\begin{eqnarray*}
	\begin{aligned}
	\a(p\lt_\l q)&=\a([\N (p)_\l q])= [\a(\N (p))_\l \a(q)]= [\N(\a(p))_\l \a(q)]=\a(p)\lt_\l \a(q),\\
     \a(p\rt_\l q)&=\a([p_\l \N(q)])= [\a (p)_\l \a(\N(q))]= [\a(p)_\l \N(\a(q))]=\a(p)\rt_\l \a(q),\\ 
     \a(p\vee_\l q)&=-\a (\N[ p_\l q])= -\N([\a(p)_\l \a(q)])=\a(p)\vee_\l \a(q).
	\end{aligned}
\end{eqnarray*}Next, we show that Eq. \eqref{eq28} holds. For any $p,q, r \in \L$, we have
	\begin{align*}
\a(p)\rt_\l(q*_\m r)&= [\a(p)_\l\N(q\lt_\m r+q\rt_\m r+q\vee_\m r)]\\&= [\a(p)_\l\N(q\lt_\m r)]+[\a(p)_\l\N(q\rt_\m r)]+[\a(p)_\l\N(q\vee_\m r)]\\&=[\a(p)_\l\N([\N q_\m r])]+[\a(p)_\l\N([q_\m \N r])]-[\a(p)_\l\N^2([q_\m r])]\\&=[\a(p)_\l [\N q_\m \N r]]\\&=[([p _\l \N q]) _{\l+\m} \a(\N r)]+[\a\N (q)_\m ([p_\l \N r])]\\&= (p\rt_\l q)\rt_{\l+\m} \a(r)+\a(q)\lt_\m (p\rt_\l r).\end{align*} 
\begin{align*}
    \a(p)\lt_\l(q\rt_\m r)&=
    [\N\a(p)_\l([q_\m \N (r)]])\\&= [([\N p_\l q])_{\l+\m} \N \a(r)]+[a(q)_\m [\N (p)_\l\N (r)]]\\&= [([\N p_\l q])_{\l+\m} \N \a(r)]+[a(q)_\m \N (p*_\l r)]\\&= (p\lt_\l q)\rt_{\l+\m} \a(r)+\a(q)\rt_\m (p*_\l r).
\end{align*}
Similarly, we can show the 3rd axiom, while the 4th axiom is given as follows.
\begin{align*}
&\a(p)\vee_\l(q*_\m r)
-\a(q)\vee_\m(p*_\l r)
-(p*_\l q)\vee_{\l+\m} \a(r)\\&
+\a(p)\lt_\l(q\vee_\m r)
-\a(q)\lt_\m(p\vee_\l r)
-(p\vee_\l q) \rt_{\l+\m}\a(r)\\
&=-\N [\a(p)_\l(q*_\m r)]
+\N[\a(q)_\m(p*_\l r)]
+\N[(p*_\l q)_{\l+\m} \a(r)]\\&
-[\N\a(p)_\l \N([q_\m r])]
+[\N\a(q)_\m \N([p_\l r])]
+[\N([p_\l q])  _{\l+\m}\N\a(r)]\\
&=-\N [\a(p)_\l([\N(q)_\m r] + [q_\m \N(r)] - \N[q_\m r])]\\&
+\N[\a(q)_\m([\N(p)_\l r] + [p_\l \N(r)] - \N[p_\l r])]\\&
+\N[([\N(p)_\l q] + [p_\l \N(q)] - \N[p_\l q])_{\l+\m} \a(r)]\\&
-\N ([\N\a(p)_\l [q_\m r]]+[\a(p)_\l \N([q_\m r])]-\N[\a(p)_\l [q_\m r]])\\& +\N ([\N\a(q)_\m [p_\l r]]+[\a(q)_\m \N([p_\l r])]-\N[\a(q)_\m [p_\l r]])\\&
+\N([\N([p_\l q])  _{\l+\m}\a(r)]+[[p_\l q] _{\l+\m}\N\a(r)]-\N[[p_\l q]_{\l+\m}\a(r)]\\
&=\N (-[\a(p)_\l[\N(q)_\m r]]+ [[p_\l \N(q)]_{\l+\m} \a(r)] +[\N\a(q)_\m [p_\l r]])
\\&+\N (- [\a(p)_\l[q_\m \N(r)]] +[[p_\l q] _{\l+\m}\N\a(r)] + [\a(q)_\m[p_\l \N(r)]])\\&
+\N([\a(q)_\m([\N(p)_\l r]] -[\N\a(p)_\l [q_\m r]] +[[\N(p)_\l q]_{\l+\m} \a(r)])
\\&+\N(-[\a(q)_\m \N[p_\l r]]- [\N[p_\l q])_{\l+\m} \a(r)]-[\a(p)_\l \N([q_\m r])])\\&
+\N (\N[\a(p)_\l [q_\m r]]-\N[[p_\l q]_{\l+\m}\a(r)]-\N[\a(q)_\m [p_\l r]])
\\& +\N ([\a(q)_\m \N[p_\l r]])
+\N([\N[p_\l q]_{\l+\m}\a(r)])
+ \N([\a(p)_\l\N[q_\m r]])\\&=0+0+0+\N(-[\a(q)_\m \N[p_\l r]]- [\N[p_\l q])_{\l+\m} \a(r)]-[\a(p)_\l \N([q_\m r])])\\&+0+\N ([\a(q)_\m \N[p_\l r]])
+\N([\N[p_\l q]_{\l+\m}\a(r)])
+ \N([\a(p)_\l\N[q_\m r]])\\&=0
\end{align*}
It completes the proof.	
\end{proof}Next, we show that Rota-Baxter operator of weight $\theta$ induces a Hom-NS-Leibniz conformal algebra.
\begin{prop}Let $\R: \L \to \L$ be a Rota-Baxter operator on the  Hom-Leibniz conformal algebra $(\L, [\cdot_\l \cdot], \a)$, then $(\L, \lt_\l,\rt_\l,\vee_\l,\a)$ is a Hom-$NS$-Leibniz conformal algebra, where
	\begin{align*}
	    p\lt_\l q= [\R(p)_\l q],\quad p\rt_\l q= [p_\l \R (q)]\quad and \quad p\vee_\l q= \theta([p_\l q]).
	\end{align*}
\end{prop}
\begin{proof}As we can show that
	\begin{eqnarray*}
	\begin{aligned}
	\a(p\lt_\l q)&=\a[\R (p)_\l q]= [\a(\R (p))_\l \a(q)]= [\R(\a(p))_\l \a(q)]=\a(p)\lt_\l \a(q),\\
     \a(p\rt_\l q)&=\a[p_\l \R (q)]= [\a (p)_\l \a(\R (q))]= [\a(p)_\l \R (\a(q))]=\a(p)\rt_\l \a(q),\\ \a(p\vee_\l q)&=\a \theta [ p_\l q]= \theta[\a(p)_\l \a(q)]=\a(p)\vee_\l \a(q).
	\end{aligned}
\end{eqnarray*}Next, we show that Eq. \eqref{eq28} holds. For any $p,q, r \in \L$, we have
	\begin{align*}
\a(p)\rt_\l(q*_\m r)&= [\a(p)_\l\R (q\lt_\m r+q\rt_\m r+q\vee_\m r)]\\&= [\a(p)_\l\R (q\lt_\m r)]+[\a(p)_\l\R (q\rt_\m r)]+[\a(p)_\l\R (q\vee_\m r)]\\&=[\a(p)_\l\R ([\R  (q)_\m r])]+[\a(p)_\l\R ([q_\m \R (r)])]+\theta[\a(p)_\l\R ([q_\m r])]\\&=[\a(p)_\l [\R  (q)_\m \R  (r)]]\\&=[([p_\l \R  (q)]) _{\l+\m} \a(\R  (r))]+[\a\R  (q)_\m ([p_\l \R  (r)])]\\&= (p\rt_\l q)\rt_{\l+\m} \a(r)+\a(q)\lt_\m (p\rt_\l r),\end{align*} 
\begin{align*}
    \a(p)\lt_\l(q\rt_\m r)&=
    [\R \a(p)_\l([q_\m \R  (r)]])\\&= [([\R  p_\l q])_{\l+\m} \R  \a(r)]+[a(q)_\m [\R  (p)_\l\R  (r)]]\\&= [([\R  p_\l q])_{\l+\m} \R  \a(r)]+[a(q)_\m \R  (p*_\l r)]\\&= (p\lt_\l q)\rt_{\l+\m} \a(r)+\a(q)\rt_\m (p*_\l r).
\end{align*}
Similarly, we can show the 3rd axiom, while the 4th axiom is given as follows.
\begin{align*}
&\a(p)\vee_\l(q*_\m r)
-\a(q)\vee_\m(p*_\l r)
-(p*_\l q)\vee_{\l+\m} \a(r)\\&
+\a(p)\lt_\l(q\vee_\m r)
-\a(q)\lt_\m(p\vee_\l r)
-(p\vee_\l q) \rt_{\l+\m}\a(r)\\
&=\theta  [\a(p)_\l(q*_\m r)]
-\theta [\a(q)_\m(p*_\l r)]
-\theta [(p*_\l q)_{\l+\m} \a(r)]\\&
+\theta[\R \a(p)_\l  [q_\m r] ]
-\theta[\R \a(q)_\m    [p_\l r] ]
-\theta[[p_\l q]_{\l+\m}\R \a(r)]\\
&=\theta  ([\a(p)_\l([\R (q)_\m r] + [q_\m \R (r)] +\theta [q_\m r]))])\\&
-\theta ([\a(q)_\m([\R (p)_\l r] + [p_\l \R (r)] +\theta  [p_\l r]))])\\&
-\theta ([([\R (p)_\l q] + [p_\l \R (q)] +\theta  [p_\l q])_{\l+\m} \a(r)])\\&
+\theta[\R \a(p)_\l  [q_\m r] ]
-\theta[\R \a(q)_\m    [p_\l r] ]
-\theta[[p_\l q]_{\l+\m}\R \a(r)]\\&
=\theta [\a(p)_\l [\R (q)_\m r]] + \theta [\a(p)_\l [q_\m \R (r)]] + \theta^2 [\a(p)_\l [q_\m r]]\\&
-\theta [\a(q)_\m[\R (p)_\l r]] -\theta [\a(q)_\m [p_\l \R (r)]] -\theta^2 [\a(q)_\m [p_\l r]]\\&
-\theta [([\R (p)_\l q])_{\l+\m} \a(r)] - \theta[[p_\l \R (q)]_{\l+\m} \a(r)] - \theta^2[ [p_\l q]_{\l+\m} \a(r)]\\&
+\theta[\R \a(p)_\l [q_\m r] ]
-\theta[\R \a(q)_\m [p_\l r] ]
-\theta[[p_\l q]_{\l+\m}\R \a(r)]\\&
=\theta [\a(p)_\l [\R (q)_\m r]]- \theta[[p_\l \R (q)]_{\l+\m} \a(r)]-\theta[\R \a(q)_\m [p_\l r] ]
\\&
+ \theta [\a(p)_\l [q_\m \R (r)]] -\theta [\a(q)_\m [p_\l \R (r)]]-\theta [[p_\l q]_{\l+\m}\R \a(r)]\\&
-\theta [\a(q)_\m[\R (p)_\l r]]  +\theta[\R \a(p)_\l [q_\m r] ]
-\theta [([\R (p)_\l q])_{\l+\m} \a(r)] \\&
  + \theta^2 [\a(p)_\l [q_\m r]]
 -\theta^2 [\a(q)_\m [p_\l r]]
- \theta^2 [[p_\l q]_{\l+\m} \a(r)]\\&=0+0+0+0=0.
\end{align*}
It completes the proof.	
\end{proof}
Now, we define the Twisted Rota-Baxter operator on Hom-Leibniz conformal algebra that yields a different class of examples on Hom-$NS$-Leibniz conformal algebras.\\
Consider that $(\M,l,r,\b)$ be a representation of a Hom-Lie conformal algebra $(\L, [\cdot_\l \cdot],\a)$. Suppose that $2$-cocycle in the cohomology complex of $(\L,[\cdot_\l \cdot],\a)$ with coefficients in $(\M,l,r,\b)$ is denoted by $\varphi_\l \in C^2_{Hom}(\L,\M)$, such that $\varphi_\l: \L\otimes \L\to \M[\l]$ and satisfying the following axioms:
\begin{align*}
    \varphi_\l (\p p,q) &= -\l \varphi_\l (p,q),\\
    \varphi_\l (p, \p q) &= (\p + \l)\varphi_\l (p,q),\\
    \a(p)_\l \varphi_\m (q,r) - \a(q)_\m \varphi_\l (p,r) - (\varphi_\l (p,q))_{\l+\m} \a(r)&\\+\varphi_{\l} (\a(p), [q_\m r])  - \varphi_\m (\a(q), [p_\l r]) - \varphi_{\l+\m} ([p_\l q], \a(r))&= 0,
\end{align*} for all $p,q,r\in \L.$
\begin{defn}
	An $\varphi$-twisted Rota-Baxter operator is $\mathbb{C}[\p]$-module homomorphism $\T:\M\to\L$ satisfying  \begin{align}\nonumber\a\T&=\T\a,\\
    \label{eq34}
	[\T (m)_\l \T (n)]&=\T(l(\T (m))_\l n+ r(m)_{\l}\T (n)+ \varphi_\l(\T (m), \T (n))) \end{align}for all $m,n\in \M$
\end{defn}
\begin{defn}\label{def2.5}Let $(\M,\beta,l,r)$ be a representation of the Hom-Leibniz conformal algebra $(\L, [\cdot_\l\cdot], \a)$. A $\mathbb{C}[\p]$-module homomorphism $\T: \M \to \L$ is referred to as a $\mathcal{O}$-operator on the Hom-Leibniz conformal algebra  with respect to its representation if the following conditions hold
	\begin{align}\nonumber \T \beta &=\alpha \T,\\ [ \T (m)_\l  \T (n)]&=\T(l(\T m)_{\lambda}(n)+r(m)_\l \T (n)).
	\end{align} for all $m,n\in \M$.
\end{defn}
\begin{ex}Any $\varphi$-twisted Rota-Baxter operator with $\varphi_\l(\T-,\T-) =0$ is simply an $\mathcal{O}$-operator on Hom-Leibniz conformal algebra.
\end{ex}
\begin{ex}Let $(\M,l,r,\b) $ be representation of a Hom-Leibniz conformal algebra $(\L,[\cdot_\l \cdot],\a)$ such that $dim(\L)=\dim(\M)$. Also for any invertible $1$-cochain $h\in C^1_{Hom}(\L,\M)$, the map $h^{-1}:=\T:\M\to\L$ is a $\varphi$-twisted Rota-Baxter operator, where $\varphi= -\delta_{HomL}(h)$.
\begin{equation}\label{eq35}
\begin{aligned}
\varphi_\l(\T (m),\T (n))=-\delta_{HomL}(h)= -(l(\T (m))_{\l} h(\T (n))+ r h(\T (m))_{\l}\T (n)- h([\T (m)_{\l} \T (n)])).
\end{aligned}
\end{equation}Applying $\T$ on the both sides of Eq. \eqref{eq35} and assuming that $\T= h^{-1}$, we get Eq. \eqref{eq34}, that is an $\varphi$-twisted Rota-Baxter identity.\end{ex}
Next, we see that the Twisted Rota-Baxter operator can also obtained by using the Nijenhuis operator $\N$ on the Hom-Leibniz conformal algebra $\L$. The following example clarifies this.
\begin{ex}\label{ex5.9} Consider a Hom-Leibniz conformal algebra $(\L, [\cdot_\l \cdot], \a)$ and a Nijenhuis operator $\N:\L\to\L$ on it. Taking into account the deformed Hom-Leibniz conformal algebra $\L_{\N}=(\L, [\cdot_\l \cdot]_{\N},\a)$ of the Proposition \ref{prop2.9}. This deformed Hom-Leibniz conformal algebra has a representation given by $l(p)_\l q = [\N (p)_\l q]$, $r(p)_\l q = [p_\l \N (q)],$ for $p\in \L_{\N}, q \in \L$. As a result, the map $\varphi_\l:\wedge^2 \L_{\N}\to \L$ given by $\varphi_\l(p,q):= -\N([p_\l q])$ is $2$-cocycle in the Hom-Leibniz conformal algebra's cohomology with the coefficients in $\L$.
\end{ex}
\begin{prop}
Consider an  $H$-twisted Rota-Baxter operator given by $\T: \M\to\L$. Then $(\M,\b, l,r)$ is a Hom-$NS$-Leibniz conformal algebra with the multiplication given by 
\begin{align*}
   m\lt_\l n =l(\T (m))_\l n, \quad m\rt_\l n =r(m)_\l \T (n), \quad  m\vee_\l n = \varphi_\l(\T (m), \T (n)),
\end{align*} for $m, n\in \M$.
\end{prop}
\begin{proof}
To satisfy the properties of Hom-NS-Leibniz conformal algebra, consider that \begin{align*}
    \b(m\lt_\l n )=\b(l(\T (m) )_\l n)=\b(l(\T (m)))_\l\b(n)= l(\b(\T (m)))_\l\b(n) =\b(m)\lt_\l\b(n).
\end{align*} Similarly, we can show that \begin{align*}
    \b(m\lt_\l n )=\b(m)\lt_\l\b(n).
\end{align*}
Next, we have \begin{align*}
    \b(m\vee_\l n )=\b \varphi_\l(\T (m), \T (n) )= \varphi_\l( \a\T (m),\a \T (n))= \varphi_\l(\T\b m,\T \b n)= \b(m)\vee_\l \b (n).
\end{align*} Further, we prove the condition in Eq. \eqref{eq28} one by one, for any $m, n, w \in \M$ 
\begin{align*}
\b(m)\rt_\l(n*_\m w)&= r(\b(m))_\l\T (n\lt_\m w+n\rt_\m w+n\vee_\m w)]\\&= r(\b(m))_\l\T (n\lt_\m w)]+r(\b(m))_\l\T (n\rt_\m w)]+r(\b(m))_\l\T (n\vee_\m r)]\\&=r(\b(m))_\l \T ([n_\m  w]_\T)\\&=r(\b(m))_\l  ([\T (n)_\m  \T (w)])\\&=r(r(m)_\l \T (n))_{\l+\m} \a(\T( w))+ l(\a\T (n))_\m (r(m)_\l \T (w))\\&= (m\rt_\l n)\rt_{\l+\m} \b(w)+\b(n)\lt_\m (m\rt_\l w).\end{align*} 
\begin{align*}
    \b(m)\lt_\l(n\rt_\m w)&=
    l(\T\b(m))_\l(r(n)_\m \T (w)))\\&= r(l(\T (m))_\l n)_{\l+\m} \T\b(w) + l(\b(n))_\m [\T(m)_\l\T (w)]]\\&
    = r(l(\T  m)_\l n)_{\l+\m} \T \b(w) + l(\b(n))_\m \T (m*_\l w)\\&= (m\lt_\l n)\rt_{\l+\m} \b(w)+\b(n)\rt_\m (m*_\l w).
\end{align*}
Similarly, we can show 3rd axiom. Since $\varphi$ is $2$-cocycle, so $$\delta_{HomL}\varphi_{\l,\m}(\T (m),\T (n),\T (w))=0.$$ Hence, we have
\begin{align*}
&\b(m)\vee_\l(n*_\m w)
-\b(n)\vee_\m(m*_\l w)
-(m*_\l n)\vee_{\l+\m} \b(w)\\&
+\b(m)\lt_\l(n\vee_\m w)
-\b(n)\lt_\m(m\vee_\l w)
-(m\vee_\l n) \rt_{\l+\m}\b(w)\\&=0.\end{align*}
It completes the proof.	
\end{proof}
\section*{Declarations}
\subsection*{Competing interests:}
The authors have no competing interests to declare that are relevant to the content of this paper.
\subsection*{Funding:}This work is funded by the Second batch of the Provincial project of the Henan Academy of Sciences (No. 241819105).
\subsection*{Availability of data and materials:}
All data is available within the manuscript.

\end{document}